\documentclass[11pt, reqno]{amsart}
\usepackage{excessive}

\title[Maximum Gonality]{On the Maximum Gonality of a Curve over a Finite Field}
\author{Xander Faber}
\author{Jon Grantham}
\address[Faber, Grantham]{Institute for Defense Analyses \\
Center for Computing Sciences \\
17100 Science Drive \\
Bowie, MD}
\author{Everett W. Howe}
\address[Howe]{Unaffiliated mathematician\\
San Diego, CA}
\date{\today}

\begin{document}
\begin{abstract}
 The gonality of a smooth geometrically connected curve over a field
 $k$ is the smallest degree of a nonconstant $k$-morphism from the
 curve to the projective line.  In general, the gonality of a curve of
 genus $g \ge 2$ is at most $2g - 2$. Over finite fields, a result of
 F.\,K.~Schmidt from the 1930s can be used to prove that the gonality
 is at most $g+1$. Via a mixture of geometry and computation, we
 improve this bound: for a curve of genus $g \ge 5$ over a finite
 field, the gonality is at most $g$. For genus $g = 3$ and $g = 4$,
 the same result holds with exactly $217$ exceptions: There are two
 curves of genus~$4$ and gonality~$5$, and $215$ curves of genus~$3$
 and gonality~$4$. The genus-$4$ examples were found in other papers,
 and we reproduce their equations here; in supplementary material, we
 provide equations for the genus-$3$ examples.
\end{abstract}
\maketitle
% keywords: curves over finite fields, gonality
% MSC 2020:
%  11G20 - curves over finite fields (primary)
%  14H51 - special divisors, gonality, ... (primary)
%  14Q05 - Computational aspects of algebraic curves (secondary)

%%%%%%%%%%%%%%%%%%%%%%%%%%%%%%%%%%%%%%%%%%%%%%%%%%%%%%%%%%%%%%%%%%%%%%%%%%%%%%%%
%%%%%%%%%%%%%%%%%%%%%%%%%%%%%%%%%%%%%%%%%%%%%%%%%%%%%%%%%%%%%%%%%%%%%%%%%%%%%%%%

\section{Introduction}
Throughout, we use the unqualified term ``curve'' to mean a smooth
complete geometrically connected variety of dimension~$1$
over a field. Let $C_{/k}$ be a curve of genus $g$ over a field
$k$. The \define{gonality} of $C$ is the minimum degree of a
nonconstant $k$-morphism $f \colon C \to \PP^1$. For the algebraically
inclined reader, the function field $\kappa(C)$ contains infinitely
many isomorphic copies of a rational function field $k(x)$, and the
gonality of $C$ is the minimum value of $[\kappa(C) \colon k(f)]$ as
$f$ varies over nonconstant functions on $C$. For the geometrically
minded reader, the gonality is the minimum degree of an effective
divisor $D$ on $C$ whose associated linear system has positive
dimension. Taking $D$ to be a canonical divisor, the Riemann--Roch
theorem shows that the gonality of $C$ is at most $2g-2$ for any curve
$C$ of genus $g \ge 2$. See \cite[Appendix]{Poonen_gonality} for
general results on gonality.

In this paper, we are interested in bounding the gonality of curves over a
finite field $\FF_q$. We will restrict our attention to curves of genus $g \ge
3$ because the gonalities of curves of genus $g = 0, 1, 2$ are known to be $1,
2, 2$, respectively. The gonality of a curve of genus $g \ge 3$ over a finite
field is at most $g+1$ by a theorem of F.\,K.~Schmidt from the 1930s
\cite[Cor.~3.1.12]{TVN_alg_geom_codes}, and it is at most $g$ if the curve
admits a rational point. See \cite[Prop.~2.1]{Faber_Grantham_GF2} for proofs of
both of these statements. For convenience, let us say that a curve of genus $g$
over a finite field is \define{excessive} if it has gonality $g+1$.  If
$C_{/\FF_q}$ is a curve of genus $g \ge 3$, and if $q > 4g^2$, then Weil's
inequality implies $C$ has a rational point. Consequently, for a given genus
$g$, there are only finitely many pairs $(g,q)$ for which an excessive curve
could possibly exist. In fact, excessive curves are substantially more rare than
this simple argument suggests.

\begin{theorem}
  \label{thm:main}
  Let $q$ be a prime power.
  \begin{itemize}
  \item There exists an excessive curve of genus~$3$ over $\FF_q$ if and only if
    $q \le 23$ or $q = 29$ or $q = 32$.
  \item There exists an excessive curve of genus~$4$ over $\FF_q$ if and only if
    $q = 2$ or $q = 3$.
  \item If $C$ is a curve of genus $g \ge 5$ over $\FF_q$, then its gonality
    is at most $g$.
  \end{itemize}
\end{theorem}

As there are only finitely many isomorphism classes of curves over $\FF_q$ of
genus~$g$, the following consequence is immediate:

\begin{corollary}
  \label{cor:finite}
  With finitely many exceptions \textup(up to isomorphism\textup), a curve over
  a finite field of genus $g \ge 2$ admits a complete linear system of positive
  dimension and degree at most~$g$.\qed
\end{corollary}

We wonder whether this corollary is best possible in the following sense:

\begin{question}
  Let $g \ge 2$ be an integer and $q$ a prime power. Does there exist a
  curve of genus~$g$ and gonality~$g$ over $\FF_q$?
\end{question}

As every genus-$2$ curve is hyperelliptic, the answer is ``yes'' in that case. A
curve with genus~$3$ and gonality~$3$ is nothing more than a smooth plane
quartic with a rational point \cite[Cor.~2.3]{Faber_Grantham_GF2}, and these are
easy to produce over any finite field. For $g = 4, 5$ and $q \le 4$, the answer
is ``yes'' by the main results of
\cite{Faber_Grantham_GF2,Faber_Grantham_GF34}. More generally, to answer the
question for genus~$4$ and $q$ odd, one could choose a nonsquare $c \in
\FF_q^\times$ and a cubic form $F \in \FF_q[x,y,z,w]$ such that the curve
\[
  C_{/\FF_q} \colon xy + z^2 - cw^2 = F(x,y,z,w) = 0
\]
is nonsingular \cite[Lemma~5.1]{Faber_Grantham_GF2}. As the space of such cubics
is 19-dimensional, surely this holds for some choice of $F$. A similar heuristic
approach using complete intersections of quadrics in $\PP^4$ applies to the case
of genus~$5$ curves \cite[Lemma~6.6]{Faber_Grantham_GF2}. We are unsure what the
answer should be for genus larger than~$5$.

The first two authors originally started this sequence of papers with the
following question:
\begin{question}
  Does the zeta function of a curve over a finite field know the gonality of the
  curve?
\end{question}
An equivalent formulation would be to ask if gonality is an isogeny class
invariant.

The short answer to these two equivalent questions is ``no''. For example,
consider the following pair of curves over~$\QQ$:
\begin{align*}
  C_1 &\colon x^4 + 4y^4 + 4z^4 + 20x^2y^2 - 8x^2z^2 + 16y^2z^2 = 0 \\
  C_2 &\colon 3v^2 = -17u^8 + 56u^7 - 84u^6 + 56u^5 - 70u^4 - 56u^3 - 84u^2 - 56u - 17.
\end{align*}
The former is a smooth plane quartic, while the latter is a hyperelliptic curve
of genus~$3$. In \cite{Howe_Q_jacobians}, the third author shows that these 
curves have isomorphic Jacobians (as unpolarized varieties). As both of these 
models have good reduction for all $p > 3$, we obtain examples of curves with 
the same zeta function and distinct gonalities over $\FF_p$ for all $p > 3$. See
also \cite[\S3]{Howe_isomorphic_jacobians} for an example over~$\FF_3$, and the
isogeny class \texttt{3.2.ac\_c\_ac} from~\cite{LMFDB} for an example
over~$\FF_2$.

In a sense, this example shows that classifying curves up to isogeny is
orthogonal to classifying them by gonality. Nevertheless, we can leverage some
information from the isogeny classification in order to learn about gonality.
We now summarize the proof of Theorem~\ref{thm:main}. For genus-$3$ curves, the
result follows almost immediately from work of Howe, Lauter, and Top on
pointless curves of small genus \cite{HLT}. For the remaining cases, our
strategy is as follows:
\begin{enumerate}
  \item Use the Riemann--Roch theorem to show that a curve of genus $g$ and
    gonality $g+1$ cannot have an effective divisor of degree $g-2$.
  \item Apply the Weil bound to rule out all but finitely many pairs $(g,q)$ for
    which there could exist an excessive curve of genus $g$ over $\FF_q$.
  \item Apply a technique of Serre and Lauter to write down a finite list of
    isogeny classes of abelian varieties that contains the Jacobian of every
    excessive curve over $\FF_q$. This limits us to the
    cases $g = 4, 5, 6, 7, 9$ and very few finite fields.
  \item For each $g < 9$, construct one or more search spaces of plane curves
    that contains a birational model of each excessive curve of genus
    $g$. Substantially reduce that search space using efficient C code, and then
    comb through the survivors using Magma.
  \item Rule out the single isogeny class with $g=9$ using techniques from the
    third author's PhD thesis.
\end{enumerate}

The first two steps will be performed in Section~\ref{sec:Weil}. We discuss the
third step in Section~\ref{sec:Lauter}. The fourth step requires slightly
different techniques depending on the genus. The relevant computations for
curves of genus at most~$5$ were performed in the earlier papers \cite{HLT} and
\cite{Faber_Grantham_GF2,Faber_Grantham_GF34}; these will be summarized in
Section~\ref{sec:345}. Curves of genus~$6$ and~$7$ will be treated in
Sections~\ref{sec:genus6} and~\ref{sec:genus7}, respectively. Finally, the
argument to rule out genus~$9$ is presented in Section~\ref{sec:genus9}.

\begin{remark}
  The work in Section~\ref{sec:Weil} is sufficient to prove the qualitative
  result in Corollary~\ref{cor:finite}. The results in Sections~\ref{sec:Lauter}
  and~\ref{sec:345} allow us to limit our remaining efforts to looking through
  $83$ isogeny classes of abelian varieties of dimensions $6$, $7$, and $9$ for
  Jacobians whose associated curves are excessive. We did locate a few potential
  needles in this haystack: at least four of the isogeny classes in 
  dimension~$7$ do contain Jacobians. However, we are able to show directly that
  none of the associated curves has gonality~$8$. With the exception of one 
  class in dimension~$6$ and one class in dimension~$9$, our results do not rule
  out the presence of Jacobians in other isogeny classes.
\end{remark}

Theorem~\ref{thm:main} and its corollary show that there are only finitely many
isomorphism classes of excessive curves. In principle, one should be able to
write them all down. This was already carried out for curves of genus~$4$ in
\cite[\S5]{Faber_Grantham_GF34}: there is exactly one such curve over $\FF_2$
and one over $\FF_3$. Defining equations can be found in
\cite[Thm.~5.4]{Faber_Grantham_GF2} and \cite[Thm.~5.3]{Faber_Grantham_GF34},
and are reproduced here in Section~\ref{sec:exhaustion}. To finish the story, we
carried out an exhaustive search for excessive curves of genus~$3$. Every such
curve arises as a pointless plane quartic. This search can be done for $q \le 5$
by brute force. For $7 \le q \le 23$, it was necessary to be more clever in
order to cut the search space of plane quartics down to something
manageable. For $q = 29$, we used isogeny class methods to show that any such
curve must be a double cover of an elliptic curve, which gives another method of
searching. And for $q=32$, we used existing literature.  The number of
isomorphism classes of excessive curves of genus~$3$ over each finite field is
given in Table~\ref{tab:excessive}. Our methods for finding all such curves are
described in greater detail in Section~\ref{sec:exhaustion}.

\begin{table}[h!]
\caption{The number of isomorphism classes of excessive curves of genus~$3$ over $\FF_q$.}
\label{tab:excessive}
  \setlength{\tabcolsep}{0.8em}  
\begin{tabular}{r@{\qquad}cc@{\quad}r@{\qquad}c}
  \toprule
  $q$ & \llap{\#}Curves && $q$ & \llap{\#}Curves \\
  \cmidrule(lr){1-2}    \cmidrule(r){4-5}
    $2$ & \pz$4$   &&      $13$ &   $11$  \\
    $3$ & \pz$8$   &&      $16$ & \pz$8$  \\
    $4$ &   $21$   &&      $17$ & \pz$7$  \\
    $5$ &   $31$   &&      $19$ & \pz$2$  \\
    $7$ &   $32$   &&      $23$ & \pz$2$  \\
    $8$ &   $39$   &&      $29$ & \pz$1$  \\
    $9$ &   $27$   &&      $32$ & \pz$1$  \\
   $11$ &   $21$   && \pz\pz\pz &         \\
  \bottomrule  
\end{tabular}
\end{table}

We close this introduction with a comment on terminology and a remark about the
software and computational resources used in this work. We view a curve $C$ as a
scheme over a base field~$k$, and by a ``point'' of $C$ we generally mean a
\define{closed point} --- that is, a morphism from $\Spec K$ to~$C$, for some
finite extension $K$ of~$k$. A closed point of degree~$d$ is a closed point 
where $K$ is a degree-$d$ extension; a rational point is a closed point of 
degree~$1$; and a quadratic (or cubic, etc.) point is a closed point of 
degree~$2$ (or~$3$, etc.). On the other hand, when we refer to a $K$-rational
point on~$C$ for some extension $K$ of~$k$, we mean a rational point on the 
base extension of $C$ to~$K$.

We wrote our implementation of Lauter's algorithm in Sage \cite{sage_9.1}. The
large searches for plane curves with special vanishing properties were written
in C. Our C code benefitted from the use of the finite field library in Flint
\cite{flint-2.6.0} as well as a number of bit-level optimizations suggested by
\cite[Ch.~2]{hackers}. The genus, point counts, and gonality of the survivors of
those searches were computed in Magma \cite{magma}. Most computations
were performed on a cluster of Intel Xeon E5-2699v3 CPUs running at $2.3$GHz at
the Center for Computing Sciences in Bowie,~MD; some Magma computations related
to enumerating excessive curves of genus~$3$ were performed on an Apple M1~Max
running at $3.2$GHz at the third author's home in San Diego,~CA. The programs we
used are available at \url{https://github.com/RationalPoint/excessive}.

%%%%%%%%%%%%%%%%%%%%%%%%%%%%%%%%%%%%%%%%%%%%%%%%%%%%%%%%%%%%%%%%%%%%%%%%%%%%%%%%
%%%%%%%%%%%%%%%%%%%%%%%%%%%%%%%%%%%%%%%%%%%%%%%%%%%%%%%%%%%%%%%%%%%%%%%%%%%%%%%%

\section{Implications of the Weil Bound}
\label{sec:Weil}

Our goal in this section is to apply the Weil bound to show that there are only
finitely many pairs $(g,q)$ for which a curve $C_{/\FF_q}$ of genus $g$ and
gonality $g+1$ could exist.

\begin{lemma}
  \label{lem:effective}
  Let $C$ be a curve of genus $g \ge 3$ over $\FF_q$. Then $C$ admits a morphism
  to $\PP^1$ of degree $g$ if and only if it has an effective divisor of degree
  $g-2$.
\end{lemma}

\begin{proof}
If $D$ is any divisor of degree~$g$, the Riemann--Roch Theorem asserts that
\[
  \dim \linsys{K-D} = \dim\linsys{D} + \deg(K-D) + 1 - g = \dim \linsys{D} - 1.
\]  
Hence, the linear system $\linsys{K-D}$ contains an effective divisor of degree 
$g-2$ if and only if $\dim \linsys{D} > 0$. The latter condition holds for some
degree-$g$ divisor $D$ if and only if $C$ admits a morphism to $\PP^1$ of 
degree~$g$.
\end{proof}

Suppose now that $C_{/\FF_q}$ is a curve of genus~$g$ and gonality $g+1$. Then
$C$ does not admit a morphism to $\PP^1$ of degree~$g$. By the lemma, $C$ has no
effective divisor of degree $g-2$. In particular, this means 
$C(\FF_{q^{g-2}}) = \varnothing$. Weil's inequality tells us that
\[
\#C(\FF_{q^r}) \ge q^r + 1 - 2gq^{r/2} \qquad (r \ge 1).
\]
It follows that $C$ has an $\FF_{q^r}$-rational point if $q^r \ge 2gq^{r/2}$, or
equivalently, if $q \ge (2g)^{2/r}$. Applying this with $r = g-2$, we conclude 
that
\[
q < (2g)^{\frac{2}{g-2}}. 
\]

Define $f(g)$ to be the function on the right in the preceding inequality. Then
$f$ is decreasing for $g \ge 3$, and $f(11) < 2$. It follows that there is no
curve of genus~$g$ and gonality~$g+1$ if $g > 10$. For $7 \le g \le 10$, we find
that $f(g) < 3$. Finally, $f(3) = 36$, $f(4) = 8$, and $f(5) \approx 4.64$. The
following proposition summarizes these findings.

\begin{proposition}
  \label{prop:bounds}
  If there exists an excessive curve of genus $g$ over $\FF_q$, then
  one of the following is true\textup{:}
  \begin{itemize}
  \item $g = 3$ and $q \le 32$\textup{;}
  \item $g = 4$ and $q \le 7$\textup{;}
  \item $g = 5$ and $q \le 4$\textup{;}
  \item $g = 6$ and $q = 2$ or $3$\textup{;} or
  \item $7 \le g \le 10$ and $q = 2$. \qed
  \end{itemize}
\end{proposition}

%%%%%%%%%%%%%%%%%%%%%%%%%%%%%%%%%%%%%%%%%%%%%%%%%%%%%%%%%%%%%%%%%%%%%%%%%%%%%%%%
%%%%%%%%%%%%%%%%%%%%%%%%%%%%%%%%%%%%%%%%%%%%%%%%%%%%%%%%%%%%%%%%%%%%%%%%%%%%%%%%

\section{Implications of Lauter's Algorithm}
\label{sec:Lauter}

Following an idea of Robinson~\cite{Robinson1964} that was used by Serre in the
context of curves over finite fields~\cite[VII.2]{Serre_Rational_Points_Book},
Lauter gave an algorithm for writing down all zeta functions of hypothetical
curves with a specified large number of points \cite{Lauter_algorithm_1998}. The
magic is that ``large'' is irrelevant to the method, and rather, one is actually
specifying zeta functions of curves with \textit{given} number of rational
points. We capitalize on this method by adding constraints on the
number of points of higher degree.

Let $C_{/\FF_q}$ be a curve of genus $g \ge 2$ over the finite field~$\FF_q$. 
Write $a_d$ for the number of closed points of degree~$d$ on the curve~$C$. An
inclusion-exclusion argument relates $a_d$ to the number of points on $C$ of a 
particular degree:
\[
\#C(\FF_{q^r}) = \sum_{d\mid r} d a_d \qquad \Longleftrightarrow\qquad
a_d = \frac{1}{d}\sum_{r\mid d} \mu\left(\frac{d}{r}\right) \#C(\FF_{q^r}).
\]

\begin{proposition}
  \label{prop:vanishing}
  Suppose that $C$ is an excessive curve of genus $g$ over $\FF_q$. Write $a_d$
  for the number of closed points of degree~$d$ on $C$. Then the conditions in
  the following table must hold.
  \begin{displaycenter}
  \setlength{\tabcolsep}{0.8em}
  \begin{tabular}{rl}
    $g$ & \textup{Vanishing conditions} \\
    \cmidrule[\lightrulewidth](lr{1em}){1-2}
     $3$ & $a_1 = 0$ \\
     $4$ & $a_1 = a_2 = 0$ \\
     $5$ & $a_1 = a_3 = 0$ \\
     $6$ & $a_1 = a_2 = a_4 = 0$ \\
     $7$ & $a_1 = a_5 = a_2 a_3 = 0$ \\
     $8$ & $a_1 = a_2 = a_3 = a_6 = 0$ \\
     $9$ & $a_1 = a_7 = a_2 a_3 = a_2 a_5 = a_3 a_4 = 0$ \\
    $10$ & $a_1 = a_2 = a_4 = a_8 = a_3 a_5 = 0$
  \end{tabular}
  \end{displaycenter}
\end{proposition}

\begin{proof}
  The argument in all cases is similar, so we describe the case $g = 7$ by way
  of example.  If $C$ is excessive, then it has gonality~$8$. In particular, it
  does not admit a morphism to $\PP^1$ of degree~$7$, and so 
  Lemma~\ref{lem:effective} implies that $C$ does not admit an effective divisor
  of degree $5$. If $C$ has a rational point $P$, then $5P$ is just such an
  effective divisor. Hence, $a_1 = 0$. 

  In what remains, if $P$ is a point of $C(\FF_{q^d})$, we write $\Pbar$ for the
  effective divisor of degree~$d$ with simple support on the Galois orbit
  of~$P$.  Now if $C$ has a point $P$ of degree~$5$, then the divisor $\Pbar$ is
  effective of degree~$5$. Hence, $a_5 = 0$. If $C$ admits a quadratic point $P$
  and a cubic point $Q$, then the divisor $\Pbar + \Qbar$ is effective of
  degree~$5$. Hence $a_2 a_3 = 0$. 
\end{proof}

The zeta function of a genus-$g$ curve $C_{/\FF_q}$ is of the form 
$\frac{P(T)}{(1-T)(1-qT)}$, where $P \in \ZZ[T]$ is a polynomial of degree~$2g$.
The \define{Weil polynomial} of $C$ is the polynomial 
$f\colonequals T^{2g}P(1/T)$. The Weil polynomial is a monic polynomial in 
$\ZZ[T]$ whose roots in the complex numbers all lie on the circle of radius
$\sqrt{q}$, and whose real roots (if any) have even multiplicity. The Weil
polynomial of $C$ is determined by the numbers $a_1, \ldots, a_g$. The 
\define{real Weil polynomial} of $C$ is the unique polynomial $h \in \ZZ[T]$
such that $f(T) = T^g h(T + q/T)$. It is monic of degree~$g$, all of its complex
roots are real numbers in the interval $[-2\sqrt{q},2\sqrt{q}]$, and it is also 
determined by $a_1, \ldots, a_g$.  Lauter's algorithm takes constraints on the 
$a_1, \ldots, a_g$ and returns real Weil polynomials for hypothetical curves 
satisfying those constraints. Using the vanishing conditions from 
Proposition~\ref{prop:vanishing} and a Sage implementation of Lauter's 
algorithm, we arrived at the following result.

\begin{theorem}
  \label{thm:Lauter_search}
  The following are true\textup{:}
\begin{itemize}
\item There is no excessive curve of genus~$4$ over $\FF_5$ or $\FF_7$.
  
\item If there is an excessive curve of genus~$6$ over $\FF_2$, then its real 
  Weil polynomial is among the following\textup{:}
  \begin{align*}
   & (T - 2)  (T + 1)  (T^2 - 2T - 2)  (T^2 - 8) \\
   & (T^2 - 8)  (T^4 - 3T^3 - 2T^2 + 7T + 1) \\
   & (T^2 - 8)  (T^4 - 3T^3 - 2T^2 + 8T - 2).
  \end{align*}

\item There is no excessive curve of genus~$6$ over $\FF_3$.  
  
\item If there is an excessive curve of genus~$7$ over $\FF_2$, then its real
  Weil polynomial is among the $79$ options given in 
  Table~\textup{\ref{table:genus7}} in Appendix~\ref{sec:genus7_classes}.

\item There is no excessive curve of genus~$8$ over $\FF_2$.

\item If there is an excessive curve of genus~$9$ over $\FF_2$, then its real
  Weil polynomial is
  \[
  (T+1)(T^4 - 2T^3 - 6T^2 + 10T + 1)^2.
  \]

\item There is no excessive curve of genus~$10$ over $\FF_2$.  \qed
\end{itemize}
\end{theorem}

%%%%%%%%%%%%%%%%%%%%%%%%%%%%%%%%%%%%%%%%%%%%%%%%%%%%%%%%%%%%%%%%%%%%%%%%%%%%%%%%
%%%%%%%%%%%%%%%%%%%%%%%%%%%%%%%%%%%%%%%%%%%%%%%%%%%%%%%%%%%%%%%%%%%%%%%%%%%%%%%%

\section{Curves of genus 3, 4, 5}
\label{sec:345}

Proposition~\ref{prop:bounds} and Theorem~\ref{thm:Lauter_search} whittled down
the list of fields $\FF_q$ for which there could exist an excessive curve of
genus $3$, $4$, or $5$, as the following table shows.
\begin{center}
  \setlength{\tabcolsep}{0.8em}
  \begin{tabular}{ccc}
    $g$ & Surviving $q$ \\
    \midrule
    $3 $& $q \le 32$ \\
    $4$ & $q \le 4\pz$ \\
    $5$ & $q \le 4\pz$
  \end{tabular}
\end{center}
The existing literature is sufficient to finish off our description.

Consider curves of genus~$3$ first. We claim that such a curve is excessive if
and only if it can be realized as a smooth plane quartic with no rational point.
For the forward implication, we observe that an excessive curve has gonality~$4$
and no rational point. In particular, it is not hyperelliptic, so its canonical
embedding is a smooth plane quartic. For the reverse, we note that a smooth
plane quartic curve is canonically embedded and hence not hyperelliptic.
Moreover, if its gonality were~$3$, then Lemma~\ref{lem:effective} asserts that
it would have a rational point. We now appeal to a result of Howe, Lauter, and 
Top~\cite{HLT}:

\begin{theorem}
  There exists a pointless smooth plane quartic over $\FF_q$ if and only if $q
  \le 23$ or $q = 29$ or $q = 32$. \qed
\end{theorem}

The remaining cases of curves of genus $4$ and $5$ were handled by the authors 
in the papers \cite{Faber_Grantham_GF2} and \cite{Faber_Grantham_GF34}:

\begin{theorem}
  There exists a curve of genus~$4$ and gonality~$5$ over $\FF_q$ if and only if
  $q \le 3$. \qed
\end{theorem}

\begin{theorem}
  There does not exist a curve of genus~$5$ and gonality~$6$ over $\FF_q$ for
  any $q$. \qed
\end{theorem}

%%%%%%%%%%%%%%%%%%%%%%%%%%%%%%%%%%%%%%%%%%%%%%%%%%%%%%%%%%%%%%%%%%%%%%%%%%%%%%%%
%%%%%%%%%%%%%%%%%%%%%%%%%%%%%%%%%%%%%%%%%%%%%%%%%%%%%%%%%%%%%%%%%%%%%%%%%%%%%%%%

\section{Curves of genus~6}
\label{sec:genus6}

Our goal for this section is to prove the following result:

\begin{theorem}
  \label{thm:genus6}
  There does not exist a curve of genus~$6$ and gonality~$7$ over $\FF_2$.
\end{theorem}

Theorem~\ref{thm:Lauter_search} showed that if such a curve exists, then its
real Weil polynomial  must be among the following options:
\begin{displaycenter}
  \renewcommand{\arraystretch}{1.2} % Give the polynomials more breathing room
  \begin{tabular}{llrrrrrr}
    Factored real Weil polynomial                && $a_1$ & $a_2$ & $a_3$& $a_4$& $a_5$& $a_6$ \\  % Chicago Manual of Style (\S2.56) recommends sentence case for table headings
    \midrule
   $(T - 2)  (T + 1)  (T^2 - 2T - 2)  (T^2 - 8)$ && $0$ & $0$ & $0$ & $0$ & $12$ & $4$ \\
   $(T^2 - 8)  (T^4 - 3T^3 - 2T^2 + 7T + 1)$     && $0$ & $0$ & $1$ & $0$ &  $8$ & $3$ \\
   $(T^2 - 8)  (T^4 - 3T^3 - 2T^2 + 8T - 2)$     && $0$ & $0$ & $2$ & $0$ &  $4$ & $1$ \\
  \end{tabular}
\end{displaycenter}

We treat the first case using off-the-shelf technology:
  
\begin{proposition}
  \label{prop:genus6case1}
  There is no curve of genus~$6$ over $\FF_2$ with $a_1 = a_2 = a_3 = a_4 = 0$. 
\end{proposition}

\begin{proof}
  Suppose otherwise, and let $C$ be such a curve.  Our application of Lauter's
  algorithm implies that the Jacobian of $C$ has real Weil polynomial
  \[
  f(T) = (T - 2) (T + 1) (T^2 - 2T - 2) (T^2 - 8).
  \]
  Write $g_1(T) = (T-2)(T+1)(T^2-8)$ and $g_2(T) = T^2 - 2T - 2$. Then 
  $f = g_1g_2$, and the Jacobian of $C$ is isogenous to the product of abelian
  varieties $A_1, A_2$ with real Weil polynomials $g_1, g_2$, respectively. 
  By~\cite[Prop.~2.8, p.~178]{HoweLauter2012}, the ``gluing exponent'' of
  $A_1$ and $A_2$ divides $2$, so by~\cite[Thm.~2.2, p.~176]{HoweLauter2012}
  there is a double cover $C \to D$ for some curve $D$ whose Jacobian is 
  isogenous to $A_1$ or $A_2$, and whose real Weil polynomial is therefore $g_1$
  or~$g_2$. A curve with either of these real Weil polynomials has a rational
  point, which would imply that $C$ has either a rational or a quadratic point. 
  This contradiction completes the proof.
\end{proof}
  
For the two remaining real Weil polynomials, we make the following critical
observation: if such a curve exists then it must have a cubic point. 

\begin{proposition}
  \label{prop:search_space}
  Suppose there exists an excessive curve $C_{/\FF_2}$ of genus~$6$. Then $C$
  admits a singular plane model of degree $7$. Moreover, this model does not 
  pass through a rational point of the plane.
\end{proposition}

\begin{proof}
  The above discussion shows that if there exists an excessive curve of 
  genus~$6$ over $\FF_2$, then it must have a cubic point. Let $D$ be the
  effective divisor of degree~$3$ that is simply supported on the Galois orbit
  of such a point. Write $K$ for a canonical divisor on $C$. Then 
  $\deg(K-D) = 7$, and we aim to show that the linear system $\linsys{K-D}$
  determines the plane model we seek.

  Riemann--Roch shows that
  \[
  \dim \linsys{K-D} = \dim \linsys{D} + 2.
  \]
  If $\dim \linsys{D} > 0$, then $\linsys{D}$ is a $g^1_3$, which would mean $C$
  has gonality at most~$3$, a contradiction. Thus $\dim \linsys{D} = 0$ and 
  $\dim \linsys{K - D} = 2$.

  We now show that $\linsys{K-D}$ is basepoint free. Since $K - D$ is defined
  over~$\FF_2$, so is the base locus. Write $E \ge 0$ for the base divisor. By
  definition, we have $\dim \linsys{K - D - E} = 2$. We also know that
  \[
    \deg(K- D - E) = 7 - \deg(E).
  \]
  If $\deg(E) > 0$, then a nonzero function in the Riemann--Roch space $L(K-D-E)$
  gives a morphism to $\PP^1$ of degree strictly smaller than $7$, which
  contradicts the fact that $C$ is excessive. Thus $\deg(E) = 0$ and $E = 0$.

  Let $\varphi \colon C \to \PP^2$ be the morphism induced by the linear system
  $\linsys{K-D}$, and let $C'$ be its image. Note that
  \[
  \deg(K-D) = 7 = \deg(C \to C') \deg(C').
  \]
  Since $\linsys{K-D}$ has dimension~$2$, the curve $C'$ cannot be a line. Hence 
  $\deg(C') = 7$, and $C'$ is birational to $C$.
  
  It remains to show that $C'$ does not pass through a rational point of the
  plane. Suppose otherwise, and let $P \in \PP^2(\FF_2)$ be such a point. Then
  we can project through $P$ to get a morphism $C' \to \PP^1$ of degree at
  most~$6$. But then the composition $C \to C' \to \PP^1$ has degree strictly
  smaller than $7$. This contradiction completes the proof.
\end{proof}

The proposition allows us to construct a search space for excessive curves of 
genus~$6$. The space of homogeneous septic polynomials has dimension
$\binom{7+2}{2} = 36$. We may insist that our polynomials do not vanish at a
rational point of the plane. Ordinarily, this is an open condition on the space.
But ``not vanishing'' is a closed condition on an $\FF_2$-vector space, so we
are able to leverage this information to cut the space down by $7$ additional
dimensions. We now describe how to do this in practice.

Suppose $F$ is a homogeneous septic polynomial over $\FF_2$ in $x,y,z$, and let
$C'$ be the plane curve it cuts out. We find that
\[
F \text{ includes the monomials $x^7, y^7$, and $z^7$.}
\]
Indeed, if $x^7$ is not present, then the point $\triple{1}{0}{0}$ lies on the
curve~$C'$. A similar argument applies to $\triple{0}{1}{0}$ and
$\triple{0}{0}{1}$. Now we can write
\[
F = x^7 + y^7 + z^7 + xyF_1(x,y) + xzF_2(x,z) + yzF_3(y,z) + xyzG(x,y,z),
\]
where
\begin{itemize}
\item $F_1,F_2,F_3$ are bivariate homogeneous polynomials of degree~$5$, and
\item $G$ is homogeneous in three variables of degree~$4$.
\end{itemize}
In order that $F(1,1,0) \ne 0$, we see that $F_1$ must have an odd number of
nonzero coefficients. Half of all bivariate polynomials have this property, so
there are $2^5$ choices for $F_1$. A similar count applies to $F_2$ and $F_3$
since $F$ cannot vanish at $\triple{1}{0}{1}$ or $\triple{0}{1}{1}$. Assuming we
have chosen such polynomials for the $F_i$, we find that in order that
$F(1,1,1) \ne 0$, $G$ must have an odd number of nonzero coefficients. The space
of polynomials of degree~$4$ has dimension $\binom{4+2}{2} = 15$, so there are
$2^{14}$ choices for $G$. It follows that there are $2^{29}$ choices for $F$ 
that do not vanish at a rational point of $\PP^2_{\FF_2}$.

We may now perform a search over all polynomials $F$ satisfying the above
constraints. We reject any polynomial $F$ that satisfies one of the following
additional properties:
\begin{itemize}
  \item $F$ vanishes at a point $P \in \PP^2(\FF_{16})$ and some partial
    derivative of $F$ does not vanish at $P$;
  \item $F$ does not vanish at either of $\triple{0}{1}{t}$ or 
    $\triple{1}{t}{t^2}$, where $t^3 + t + 1 = 0$.
\end{itemize}
The first property asserts that there is a smooth quadratic or quartic point $P$
on $C' = \{F = 0\}$. If that were to occur, then the normalization $C$ would 
have a quadratic or quartic point, in violation of 
Proposition~\ref{prop:vanishing}. We know that a hypothetical excessive curve of
genus~$6$ over $\FF_2$ must have a cubic point. Pushing that point down to the
plane model $C'$ would give a cubic point or a rational point. We have already
ruled out the possibility of the latter. The group $\PGL_3(\FF_2)$ acts
on~$\PP^2$. There are two orbits of cubic points: one contains
$\triple{0}{1}{t}$ and the other contains $\triple{1}{t}{t^2}$. So we may insist
that our curve contains at least one of these points, which justifies the second
property above.

We wrote C code to execute our search for homogeneous polynomials of degree~$7$
satisfying the above constraints. The search required about $2$ minutes on a
single CPU, and it found $110{,}770$ polynomials. We passed these polynomials
through a Magma script that looked for irreducible plane curves of genus~$6$ 
with no point of degree $1$, $2$, or~$4$. This took around $3$ minutes. No curve
survived.

The results of this search, combined with Proposition~\ref{prop:genus6case1},
provide a proof of Theorem~\ref{thm:genus6}.

%%%%%%%%%%%%%%%%%%%%%%%%%%%%%%%%%%%%%%%%%%%%%%%%%%%%%%%%%%%%%%%%%%%%%%%%%%%%%%%%
%%%%%%%%%%%%%%%%%%%%%%%%%%%%%%%%%%%%%%%%%%%%%%%%%%%%%%%%%%%%%%%%%%%%%%%%%%%%%%%%

\section{Curves of genus 7}
\label{sec:genus7}

We apply the same strategy as in the previous section, but there is one annoying
wrinkle to be ironed out. If $C$ is an excessive curve of genus~$7$ with an
effective divisor of degree~$4$, then $C$ can be realized as a singular plane
curve of degree~$8$. We treat this case in Section~\ref{sec:genus7_4}. Looking 
at Table~\ref{table:genus7} in Appendix~\ref{sec:genus7_classes}, we find that
all but three of the real Weil polynomials have $a_2 > 0$ or $a_4 > 0$. The
exceptions --- entries $3$, $4$, and $6$ --- cannot support an effective divisor
of degree~$4$. However, each of the exceptions has $a_3 > 0$. An excessive curve
of genus~$7$ with an effective divisor of degree~$3$ can be realized as a 
singular plane curve of degree~$9$. That case will be handled in
Section~\ref{sec:genus7_3}.

Combining Theorems~\ref{thm:genus7_3_to_go} and~\ref{thm:genus7_special} below,
we will obtain the desired result:

\begin{theorem}
  There is no curve of genus $7$ and gonality $8$ over $\FF_2$. 
\end{theorem}

%%%%%%%%%%%%%%%%%%%%%%%%%%%%%%%%%%%%%%%%%%%%%%%%%%%%%%%%%%%%%%%%%%%%%%%%%%%%%%%%

\subsection{An effective divisor of degree~4}
\label{sec:genus7_4}

We begin with some geometry.

\begin{proposition}
  Suppose there exists an excessive curve $C_{/\FF_2}$ of genus~$7$ with an 
  effective divisor of degree~$4$. Then $C$ admits a singular plane model of
  degree~$8$. Moreover, this model does not pass through a rational point of the
  plane.
\end{proposition}

\begin{proof}
  Let $K$ be a canonical divisor on $C$ and $D$ an effective divisor of
  degree~$4$. An argument virtually identical to the one used to prove
  Proposition~\ref{prop:search_space} applies to show that the linear system 
  $\linsys{K - D}$ is basepoint free of degree~$8$ and dimension~$2$.  If we 
  write $C'$ for the image of $C$ under the morphism $C \to \PP^2$ determined by 
  $\linsys{K-D}$, then we find that
  \[
  \deg(K-D) = 8 = \deg(C \to C') \deg(C').
  \]
  We must argue that $\deg(C') = 8$.

  Since $\dim \linsys{K-D} = 2$, we find that $\deg(C') > 1$.  Write 
  $f \colon C \to \tilde C'$ for the induced morphism from $C$ to the 
  normalization $\tilde C'$ of $C'$. If the degree of $C'$ were~$2$, then 
  $C' = \tilde C'$ would be a smooth rational curve and the morphism $f$ would
  have degree~$4$. But the gonality of $C$ is~$8$, so this cannot occur.  If 
  instead the degree of $C'$ were~$4$, then $\deg(f) = 2$. Writing $g'$ for the 
  genus of $\tilde C'$, we see that $g' \le \frac{(4-1)(4-2)}{2} = 3$. If
  $g' \le 2$, then the gonality of $\tilde C'$ is at most $2$
  \cite[Prop.~2.1]{Faber_Grantham_GF2}. It would then follow that $C$ has 
  gonality at most $2\deg(f) = 4$, a contradiction. Therefore, $g' = 3$ and 
  $C' = \tilde C'$ is a smooth plane quartic. If $C'$ has a rational point, then
  it has gonality~$3$ \cite[Cor.~2.3]{Faber_Grantham_GF2}. Again we arrive at a
  contradiction because the gonality of $C$ would be at most $2\cdot 3 = 6$. 
  Thus $C'(\FF_2) = \varnothing$.

  To summarize, we are now assuming that $C'$ is a smooth plane quartic with no
  rational point, and $f \colon C \to C'$ has degree~$2$.  A direct computer
  search shows that there are $4$ isomorphism classes of smooth pointless plane
  quartics over $\FF_2$. Table~\ref{tab:plane_quartics} gives representative
  equations for these four classes, their number of points over $\FF_2, \FF_4,$
  and $\FF_8$, and their real Weil polynomials. By a result of Tate
  \cite[Thm.~1(b)]{Tate_Abelian_Varieties_Finite_Fields_1966}, the real Weil
  polynomial for $\tilde C'$ divides that of $C$. Looking through all of the
  entries in Appendix~\ref{sec:genus7_classes}, we find that none of those
  polynomials is divisible by any of the polynomials in
  Table~\ref{tab:plane_quartics}. We conclude that $C'$ cannot have degree~$4$.

  Thus, we have shown that $C'$ has degree~$8$, and the morphism $C \to C'$ is
  birational. To complete the proof, we note that $C'$ cannot pass through a
  rational point; otherwise we could project through that point to get a
  morphism $C \to C' \to \PP^1$ of degree strictly smaller than~$8$.
\end{proof}

\begin{table}[h!]
\renewcommand{\arraystretch}{1.2} % Give the polynomials more breathing room. 
\caption{The four isomorphism classes of pointless smooth plane quartic curves
  over $\FF_2$, along with their real Weil polynomial and numbers of
  closed points.}
\label{tab:plane_quartics}
\begin{tabular}{lrrrrr}
  \toprule
  Curve equation & Real Weil polynomial && $a_1$ & $a_2$ & $a_3$  \\
  \midrule
  $0 = x^{4} + x y^{3} + y^{4} + x y z^{2} + x z^{3} + y z^{3} + z^{4}$     & $(T-2)(T^2-T-5)$           &&    $0$ & $1$ & $1$ \\
  $0 = x^{4} + x y^{3} + y^{4} + x^{2} z^{2} + x y z^{2} + y z^{3} + z^{4}$ & $(T-2)(T^2-T-4)$           &&    $0$ & $2$ & $2$ \\
  $0 = x^{4} + x y^{3} + y^{4} + x^{3} z + x y z^{2} + y z^{3} + z^{4}$     & $T^3 - 3T^2 - 4T + 13$     &&    $0$ & $0$ & $1$ \\
  $0 = x^{4} + x^{2} y^{2} + y^{4} + x^{2} y z + x y^{2} z \ + $            & \multirow{2}{*}{$(T-1)^3$} &&   \multirow{2}{*}{$0$} & \multirow{2}{*}{$7$} & \multirow{2}{*}{$8$} \\
    \hspace{8em}$ x^{2} z^{2} + x y z^{2} + y^{2} z^{2} + z^{4}$ &&&&&\\
  \bottomrule  
\end{tabular}
\end{table}

We now spell out the differences between the genus-$6$ search and the related
genus-$7$ search. Octic homogeneous polynomials that do not vanish at any of the
rational points of the plane have the form
\[
F = x^8 + y^8 + z^8 + xyF_1(x,y) + xzF_2(x,z) + yzF_3(y,z) + xyzG(x,y,z),
\]
where
\begin{itemize}
  \item $F_1,F_2,F_3$ are bivariate homogeneous polynomials of degree~$6$, 
  \item $G$ is homogeneous in three variables of degree~$5$, and
  \item the $F_i$ and $G$ each has an odd number of nonzero coefficients.
\end{itemize}
There are $2^{38}$ polynomials satisfying these constraints. We loop over all
such polynomials and \textit{reject} any that satisfies one of the following
additional properties:
\begin{itemize}
  \item $F$ vanishes at a point $P \in \PP^2(\FF_{2^5})$ and some partial
    derivative of $F$ does not vanish at $P$;
  \item $F$ does not vanish at any of the points $\triple{0}{1}{s}$, 
    $\triple{1}{s}{s^2}$, or $\triple{0}{1}{s^2 + s}$, where $s^4 + s + 1 = 0$.
\end{itemize}
The first property asserts that the plane model $C'$ has a smooth quintic point,
which would mean $C$ has a smooth quintic point, a contradiction. By hypothesis,
we know $C$ has an effective divisor of degree~$4$, and hence a quadratic or
quartic point. It follows that $C'$ must pass through such a point. (Note that
the image of a quartic point may be a quadratic point on $C'$.) There are two
$\PGL_3(\FF_2)$-orbits of quartic points and one orbit of quadratic points;
representatives are $\triple{0}{1}{s}$, $\triple{1}{s}{s^2}$, and
$\triple{0}{1}{s^2 + s}$, where $s^4 + s + 1 = 0$.

We divided the space of $2^{38}$ such polynomials into $64$ tiles and 
distributed the search to $64$ CPUs.  The entire search required about $70$ 
minutes of wall time. Each tile generated around $2.2$ million polynomials for
further investigation.  We ran $64$ instances of our Magma script, one for each
tile; these required between $3$ and $4$ hours to complete. Out of the entire
pool of approximately $140$ million polynomials, only $606$ describe irreducible
genus-$7$ curves that have no quintic point and that do not simultaneously have
a quadratic and a cubic point. More precisely, we found
\begin{itemize}
  \item $248$ curves with $a_d$-sequence   $(0,2,0,3,0,3,12)$, and
  \item $358$ curves with $a_d$-sequence $(0,2,0,6,0,4,16)$.
\end{itemize}
Each of these curves has gonality~$4$. We determined this by searching for
divisors $D$ of degree~$2$ or~$4$ with $\dim \linsys{D} > 0$; such a divisor
must be a linear combination of Galois orbits of quadratic or quartic points. We
used Magma to verify the existence of these divisors and to compute the 
dimension of the relevant Riemann--Roch spaces. We summarize our findings thus
far:

\begin{theorem}
  \label{thm:genus7_3_to_go}
  If there is an excessive curve of genus~$7$ over $\FF_2$, then its real Weil
  polynomial is one of the following options from
  Table~\textup{\ref{table:genus7} in} Appendix~\ref{sec:genus7_classes}\,\textup{:}
  \begin{displaycenter}
  \renewcommand{\arraystretch}{1.2} % Give the polynomials more breathing room
  \begin{tabular}{clcrrrrrrr}
    \textup{No.} & \textup{Factored real Weil polynomial} && $a_1$ & $a_2$ & $a_3$& $a_4$& $a_5$& $a_6$& $a_7$  \\
    \midrule
$3$ & $T  (T^{6} - 3 T^{5} - 12 T^{4} + 39 T^{3} + 27 T^{2} - 126 T + 57)$      && $0$ & $0$ & $3$ & $0$ & $0$ & $13$ &  $9$ \\
$4$ & $T^{7} - 3 T^{6} - 12 T^{5} + 40 T^{4} + 24 T^{3} - 132 T^{2} + 75 T + 1$ && $0$ & $0$ & $4$ & $0$ & $0$ &  $9$ & $10$ \\
$6$ & $(T^{3} - T^{2} - 7 T + 6)  (T^{4} - 2 T^{3} - 7 T^{2} + 14 T - 2)$       && $0$ & $0$ & $5$ & $0$ & $0$ &  $9$ & $12$ \\
  \end{tabular}
  \end{displaycenter}
\end{theorem}

%%%%%%%%%%%%%%%%%%%%%%%%%%%%%%%%%%%%%%%%%%%%%%%%%%%%%%%%%%%%%%%%%%%%%%%%%%%%%%%%

\subsection{An effective divisor of degree~3}
\label{sec:genus7_3}

Again, we begin with some geometry.

\begin{proposition}
  \label{prop:degree9}
  Suppose there exists an excessive curve $C_{/\FF_2}$ of genus~$7$ with an 
  effective divisor of degree~$3$. Then $C$ admits a singular plane model of
  degree~$9$. Moreover, this model does not pass through a rational point of the
  plane.
\end{proposition}

\begin{proof}
  Let $D$ be an effective divisor of degree~$3$ and let $K$ be a canonical
  divisor. Then $\dim \linsys{K-D} \ge \dim \linsys{K} - 3 = 3$. If this
  dimension were larger than~$3$, we would have
  \[
    \dim \linsys{K - 2D} \ge \dim \linsys{K-D} - \deg(D) \ge 1,
  \]
  which means $C$ admits a nonconstant morphism to $\PP^1$ of degree 
  $\deg(K-2D) = 6$, a contradiction. We conclude that $\linsys{K-D}$ has 
  degree~$9$ and dimension~$3$.

  Now we argue that $\linsys{K-D}$ is basepoint free. Let $E \ge 0$ be its base
  divisor. Then $\dim \linsys{K - D - E} = 3$ and $\deg(K-D-E) = 9 - \deg(E)$. 
  If $\deg(E) \ge 2$, then $L(K-D-E)$ contains a nonconstant rational function
  of degree at most~$7$, a contradiction. If $\deg(E) = 1$, then $E$ is 
  supported at a rational point, which an excessive curve cannot have. Thus, 
  $E = 0$.

  The linear system $\linsys{K-D}$ determines a morphism $C \to \PP^3$. 
  Composing this with projection through a rational point of $\PP^3$ (which does
  not lie on~$C$), we obtain a morphism to $\PP^2$. Write $C'$ for the image in
  the plane. Then
  \[
  \deg(K-D) = 9 = \deg(C \to C')\deg(C').
  \]
  The curve $C'$ does not lie in a line, so its degree is bigger than~$1$. If
  $\deg(C') = 3$, then $C'$ has genus $0$ or~$1$, and so gonality at most~$2$.
  But then $\deg(C \to C') = 3$, and composing $C \to C'$ with a morphism to 
  $\PP^1$ of minimum degree would show that $C$ has gonality at most~$6$, a
  contradiction. So $C'$ has degree~$9$, as desired.

  Finally, suppose that $C'$ passes through a rational point $P$ of the
  plane. Since $C$ has no rational point, and since $C \to C'$ is the
  normalization morphism, we see that $P$ cannot be smooth. Suppose that $P$ has
  multiplicity $r \ge 2$ on $C'$. Then linear projection through $P$ gives rise
  to a morphism $C \to \PP^1$ of degree $9-r \le 7$, which contradicts the fact
  that $C$ has gonality~$8$. We conclude that $C'$ cannot pass through a rational
  point.
\end{proof}

Following the proposition, our first attempt might be to search through plane
curves defined by homogeneous polynomials of degree~$9$. The space of such
polynomials has $\FF_2$-dimension $\binom{9+2}{2} = 55$. As in the case of
degree-$6$ curves, we can avoid looking at polynomials that vanish at a rational
point of the plane, which cuts the space down to $2^{48}$ polynomials. This is
much too large to be efficiently searched, so we need an additional improvement.
We will leverage the fact that curves in the three remaining isogeny classes 
have at least three cubic points.

\begin{proposition}
  \label{prop:genus7_special}
  Suppose there exists an excessive curve $C_{/\FF_2}$ of genus~$7$ with at
  least three distinct closed points of degree~$3$. Write $\FF_8 = \FF_2(t)$,
  where $t^3 + t + 1 = 0$. Then $C$ is birational to a plane curve of degree~$9$
  that passes through no rational point of the plane, and that satisfies at 
  least one of the following\textup{:}
  \begin{itemize}
  \item $C$ has a triple point at $\triple{0}{1}{t}$\textup{;}
  \item $C$ has a double point at $\triple{0}{1}{t}$ and passes through
            $\triple{0}{1}{t+1}$\textup{;} or
  \item $C$ passes through $\triple{0}{1}{t}$ and $\triple{1}{0}{t}$.
  \end{itemize}
\end{proposition}

\begin{proof}
By hypothesis, $C$ has three effective divisors $D_1,D_2,D_3$ of degree~$3$ with
pairwise disjoint support. Let $K$ be an effective canonical divisor. The proof 
of Proposition~\ref{prop:degree9} shows that $\dim \linsys{K - D_1} = 3$. For
$j = 1,2,3$, we also see that
\[
   \dim \linsys{K - D_1 - D_j} = 0,
\]
since otherwise this linear system would give a morphism to $\PP^1$ of
degree~$6$. It follows that there are unique effective divisors $E_1, E_2, E_3$
of degree~$6$ such that
\begin{align*}
  K - D_1 &\sim D_1 + E_1 \\
          &\sim D_2 + E_2 \\
          &\sim D_3 + E_3.
\end{align*}
Let $f_1, f_2, f_3 \in L(K-D_1)$ be rational functions corresponding to these
three linear relations; the $f_i$ are unique up to a constant. 

Suppose first that the $f_i$ generate a subspace of $L(K-D_1)$ of dimension~$1$.
Then there is a single function $f$ on $C$ with
\[
\div(f) = 2D_1 + E_1 - K = D_1 + D_2 + E_2 - K = D_1 + D_3 + E_3 - K.
\]
Since the $D_i$'s have disjoint support, we see that $\div(f) = 2D_1 + D_2 +
D_3 - K$. Using $f$ as one of the coordinate functions for our map $C \to
\PP^2$, we see that there is a line $\ell \subset \PP^2$ that passes through
the images of $D_1$, $D_2$, and $D_3$. Let $C'$ be the image of $C$ in $\PP^2$,
and let $P_1, P_2, P_3$ be the images of the cubic points in $C'$ corresponding
to $D_1, D_2, D_3$, respectively. None of the $P_i$ is rational since $C'$
cannot pass through a rational point of the plane. Moreover, the $P_i$ cannot
all be distinct, for otherwise the line $\ell$ would contain three distinct
cubic closed points. (An $\FF_2$-rational line contains three rational points
and two cubic closed points.) Thus, either $C'$ has a triple point at $P_1$, or
$C'$ has a double point at $P_1$ and a second cubic point on the
line~$\ell$. Since $\PGL_3(\FF_2)$ acts transitively on the set of~14 cubic
closed points of the plane that lie on some $\FF_2$-rational line, we may
assume without loss that we are in one of the first two cases of the
proposition.

Now suppose that the $f_i$ generate a subspace of $L(K-D_1)$ of dimension at
least~$2$. We use two linearly independent $f_i$ as coordinate functions for our
map $C \to \PP^2$. Keeping the notation of the previous paragraph, it follows
there are distinct lines $\ell_1$ and $\ell_2$ containing cubic points $P_1$ and
$P_2$ of $C'$. After applying a linear transformation of the plane, we may
assume that our lines are $\ell_1 = \{x = 0\}$ and $\ell_2 = \{y = 0\}$. The
subgroup of $\PGL_3(\FF_2)$ that fixes these two lines is given by
\[
\begin{pmatrix} 1 & 0 & 0 \\ 0 & 1  & 0 \\ * & * & 1 \end{pmatrix} \cong \ZZ/2\ZZ \times \ZZ/2\ZZ.
\]
One can verify directly that this group acts faithfully on the set of $4$ cubic
closed points on $\ell_1 \cup \ell_2$. It follows that there is some
transformation mapping $P_1$ to the Galois orbit of $\triple{0}{1}{t}$ and $P_2$
to the Galois orbit of $\triple{1}{0}{t}$. This completes the proof.
\end{proof}

The preceding proposition allows us to restrict our attention to homogeneous
polynomials of degree~$9$ over $\FF_2$ with certain extra vanishing conditions.
In each of the three cases, we obtain a number of linear constraints on the
coefficients of our polynomials: $18$, $12$, and $6$, respectively. Starting
with a space of dimension~$55$, this means we have three search spaces to
consider, consisting of $2^{37}$, $2^{43}$, and $2^{49}$ polynomials, 
respectively. Using linear algebra, we can produce three explicit bases for the 
spaces of polynomials with these vanishing conditions. Unlike the basis of
monomials used in Section~\ref{sec:genus7_4}, a naive choice of new basis may
not be well conditioned for efficiently searching for polynomials that do not
vanish at a rational point. For completeness, we write down explicit bases of
polynomials that \emph{are} well conditioned in Appendix~\ref{sec:bases} and
argue that they have the properties we want. In the first two cases, avoiding
polynomials that vanish at a rational point cuts $5$ dimensions off the search
space; in the third case it cuts $7$ dimensions off the search space. In 
summary, this means we will have three searches consisting of $2^{32}$, 
$2^{38}$, and $2^{42}$ polynomials, respectively.

\begin{table}
\caption{Search resources and wall time for the three cases of Proposition~\ref{prop:genus7_special}.}
\label{tab:summary}
\setlength{\tabcolsep}{0.8em}
\begin{tabular}{cccrcc}
  \toprule
  Case &   CPUs  &                     C search time &        Survivors &                Magma search time & Survivors \\
  \midrule
   $1$ & $\pz 1$ & \phantom{$11$h $45$m}\llap{$37$m} &      $162{,}552$ & \phantom{$1$h $44$m}\llap{$43$m} &      $24$ \\
   $2$ &    $64$ & \phantom{$11$h $45$m}\llap{$54$m} &  $5{,}314{,}648$ & \phantom{$1$h $44$m}\llap{$28$m} &   $\pz 0$ \\
   $3$ &    $64$ &                       $11$h $45$m & $75{,}877{,}946$ &                      $1$h $44$m  &      $40$ \\
  \bottomrule
\end{tabular}
\end{table}

Table~\ref{tab:summary} describes the search resources, wall time, and number of
survivors for the three searches we performed, one for each case of
Proposition~\ref{prop:genus7_special}. As expected, the C search time for
Cases~$2$ and~$3$ are approximately $64\times$ and $1024\times$ that of 
Case~$1$, respectively. For Case~$1$, the $24$ curves that survived the Magma 
search were equally split between the first and third isogeny classes of
Theorem~\ref{thm:genus7_3_to_go}. For Case~$3$, twelve of the curves that
survived fell into the first isogeny class of Theorem~\ref{thm:genus7_3_to_go},
while the remainder fell into the third isogeny class of the theorem.  Using 
Magma, we verified that all $64$ of the survivors have gonality~$6$ by producing
an effective divisor $D$ of degree~$6$ with $\dim \linsys{D} > 0$. In summary:

\begin{theorem}
  \label{thm:genus7_special}
  There is no curve of genus~$7$ and gonality~$8$ over $\FF_2$ that admits at
  least~$3$ closed points of degree~$3$. \qed
\end{theorem}

%%%%%%%%%%%%%%%%%%%%%%%%%%%%%%%%%%%%%%%%%%%%%%%%%%%%%%%%%%%%%%%%%%%%%%%%%%%%%%%%
%%%%%%%%%%%%%%%%%%%%%%%%%%%%%%%%%%%%%%%%%%%%%%%%%%%%%%%%%%%%%%%%%%%%%%%%%%%%%%%%

\section{Genus 9}
\label{sec:genus9}

Theorem~\ref{thm:Lauter_search} shows that if there is an excessive curve $C$ of
genus~$9$ over $\FF_2$, then its real Weil polynomial and counts of closed 
points of various degrees are as below.
\begin{displaycenter}
\begin{tabular}{ccrrrrrrrrrr}
  Factored real Weil polynomial          && $a_1$ & $a_2$ & $a_3$& $a_4$& $a_5$& $a_6$& $a_7$& $a_8$& $a_9$\\
  \cmidrule[\lightrulewidth]{1-1}\cmidrule[\lightrulewidth]{3-11}
  $(T+1)(T^4 - 2T^3 - 6T^2 + 10T + 1)^2$ &&  $0$  &   $4$ &  $0$ &  $0$ &  $0$ &  $8$ &  $0$ & $18$ & $64$ \\
\end{tabular}
\end{displaycenter}
In particular, note that $C$ has a closed point of degree~$6$.  If we took the
tack of Section~\ref{sec:genus6}, we could show that a hypothetical excessive
curve of genus~$9$ over $\FF_2$ admits a singular plane model of degree~$10$ 
with no rational point. There are $2^{59}$ such curves that we would need to
search. Even after imposing vanishing conditions at the four quadratic points,
we would not be able to shrink the search space below about $2^{50}$
polynomials. This is far too large for us to handle, so instead we will use a
different method: By analyzing the principally polarized varieties with the
given real Weil polynomial, we can show that there is no genus-$9$ curve over
$\FF_2$ (of any gonality) with the given real Weil polynomial. Our argument uses
results and techniques from~\cite{Howe1995}, \cite{Howe_Lauter_2003},
and~\cite{HoweLauter2012}; an expository overview of many of these results can
be found in~\cite{Howe_Serre_proceedings}.

\begin{theorem}
\label{T:genus9}
Let $h_1 = T + 1$ and let $h_2 = T^4 - 2T^3 - 6T^2 + 10T + 1$. Then there is no 
curve over $\FF_2$ whose real Weil polynomial is equal to $h_1 h_2^2$.
\end{theorem}

\begin{proof}
Let $E_{/\FF_2}$ be the unique elliptic curve with real Weil polynomial~$h_1$.
Suppose there were a curve $C_{/\FF_2}$ with real Weil polynomial $h_1 h_2^2$. 
Since the resultant of the two polynomials $h_1$ and $h_2$ is~$-12$, we see 
from~\cite[Lemma~7]{Howe_Lauter_2003} (or its 
generalization~\cite[Lemma~2.3]{HoweLauter2012}; see 
also~\cite[Props.~3.7 and~3.10]{Howe_Serre_proceedings}) that there is an
abelian variety $A_{/\FF_2}$ with real Weil polynomial~$h_2^2$ and a finite 
$12$-torsion group scheme $\Delta$ such that the Jacobian of $C$ fits into an 
exact sequence
\[
0 \to \Delta \to E\times A \to \Jac C \to 0,
\]
and where the induced maps $\Delta\to E$ and $\Delta\to A$ are both embeddings.
Furthermore, the pullback to $E \times A$ of the canonical principal 
polarization of $\Jac C$ is the product $\lambda \times \mu$ of a polarization
$\lambda$ of $E$ and a polarization $\mu$ of $A$, and the images of $\Delta$ in
$E$ and in $A$ are the kernels of $\lambda$ and $\mu$. 

If we let $\lambda_0$ be the canonical principal polarization of~$E$, then every
polarization of $E$ is of the form $d\lambda_0$ for a positive integer~$d$. 
Since the group scheme $\Delta$ is killed by $12$, it follows that $\Delta$ is
isomorphic to $E[d]$ for a divisor $d$ of~$12$.

From~\cite[Lemma~4.3]{HoweLauter2012} (see
also~\cite[Prop.~4.13]{Howe_Serre_proceedings}) we see that there must be a
nonconstant degree-$d$ map from $C$ to~$E$. It follows that the number of closed
points of degree at most $d$ on $C$ must be at least the number of rational
points of~$E$, because each rational point of $E$ has at least one closed point
of $C$ lying over it, and the degree of each such closed point is at most~$d$. 
Since $E$ has $5$ rational points, while $C$ has no closed points of degree~$1$,
$3$, or~$4$ and only $4$ of degree~$2$, we see that $d>4$. That leaves open only
two possibilities: $d=6$ or $d=12$.

We claim that neither of these possibilities can occur, because there is no
polarization of $A$ whose kernel is isomorphic to either $E[6]$ or $E[12]$. To 
prove this, we use the results of~\cite{Howe1995}, which give us information on
the group schemes that occur as kernels of polarizations of abelian varieties 
in a given ordinary isogeny class. Given a finite multiset of finite simple
group schemes, the results tell us whether there is a polarization of a variety
in the given isogeny class such that the composition series of the kernel of the
polarization is equal to the given multiset.

Let $\pi$ denote the Frobenius endomorphism of $A$ and let $\pibar$ denote its
dual isogeny (the \define{Verschiebung}). Honda--Tate theory~\cite{Tate1971} 
shows that the minimal polynomial of the endomorphism $\pi+\pibar$ is~$h_2$, and
since $\pibar = 2/\pi$, we find that the minimal polynomial of $\pi$ is 
$f\colonequals T^4 \, h_2(T + 2/T)$. Let $R$ denote the subring 
$\ZZ[\pi,\pibar]$ of $\End A$. Then $K\colonequals R\otimes\QQ$ is the number
field defined by~$f$, and we can view $R$ as an order in this field. The field
$K$ is a CM~field, and its maximal real subfield $\Kplus$ is $\QQ(\pi+\pibar)$.
We let $\Rplus$ denote the subring $\ZZ[\pi+\pibar]$ of~$R$.

Let $G$ be a finite group scheme that is a subgroup scheme of an abelian variety
isogenous to~$A$, so that the Frobenius and Verschiebung endomorphisms $F$ and 
$V$ of $G$ satisfy $f(F) = f(V) = 0$. We associate to $G$ a finite $R$-module 
$\Mod_R G$ as follows. First, we write $G$ as a product 
$G_\rr\times G_\rl\times G_\lr$,  where $G_\rr$ is a reduced group scheme with 
reduced dual, $G_\rl$ is a reduced group scheme with local dual, and $G_\lr$ is
a local group scheme with reduced dual. (Every finite group scheme over a finite
field can be written as a product of three such group schemes together with 
another factor that is local with local dual (\cite[p.~17]{Manin1963}, 
\cite[Cor., p.~52]{Waterhouse1979}), but for group schemes that can be embedded
in an ordinary abelian variety, the local-local factor is trivial; this follows
from the fact that, geometrically, the $p$-torsion of an ordinary abelian
variety is isomorphic to a power of the reduced-local group scheme 
$\boldmu_p$~\cite[\S2]{Deligne1969}.) Let $k$ be an algebraic closure 
of~$\FF_2$. Let $M_\rr$ and $M_\rl$ be the finite $R$-modules whose underlying
abelian groups are $G_\rr(k)$ and $G_\rl(k)$, respectively, and whose $R$-module
structure is such that $\pi$ acts as $F$ and $\pibar$ acts as~$V$. Let 
$\hat{G}_\lr$ be the Cartier dual of $G_\lr$, and let $M_\lr$ be the finite 
$R$-module whose underlying abelian group is $\hat{G}_\lr(k)$ and where $\pi$
acts as $V$ and $\pibar$ acts as~$F$. Finally, we take 
$\Mod_R G = M_\rr\oplus M_\rl \oplus M_\lr$.

The \define{Grothendieck group} $\calG(S)$ of an order $S$ in a number field is 
the group obtained from the free abelian group on finite $S$-modules by dividing
out by the subgroup generated by the expressions $[M'] - [M] - [M'']$, for every 
exact sequence $0\to M\to M' \to M''\to 0$ of finite $S$-modules. The group 
$\calG(S)$ is free, with a basis given by the classes of the simple $S$-modules.

The main result of~\cite{Howe1995} tells us exactly which elements of the 
Grothendieck group of $R$ contain the classes $[\Mod_R G]$ of the $R$-modules 
corresponding to kernels $G$ of polarizations of varieties isogenous to~$A$.
(Such classes are called \define{attainable}.) To state the result, we need to 
define a certain quotient of the group $\calG(\Rplus)$.

Every finite $R$-module is also a finite $\Rplus$-module, and an exact sequence
of $R$-modules is also an exact sequence of $\Rplus$-modules. Thus there is a
homomorphism $N_{R/\Rplus}\colon\calG(R)\to\calG(\Rplus)$. Also, for every 
nonzero element $\alpha$ of $\Rplus$ we have a finite $\Rplus$-module 
$\Rplus/\alpha\Rplus$. We let $\calB(R)$ denote the quotient of $\calG(\Rplus)$ 
by the subgroup generated by the image of $N_{R/\Rplus}$ and by the classes of 
the modules $\Rplus/\alpha\Rplus$, where we let $\alpha$ range over the 
\define{totally positive} elements of $\Rplus$ --- that is, the elements that
are positive under every embedding of $\Rplus$ into the real numbers. Since 
$R = \Rplus[\pi]$ is free of rank~$2$ as an $\Rplus$-module, if we take any
finite $\Rplus$-module $M$ we find that 
$N_{R/\Rplus}([M\otimes_{\Rplus} R]) = 2[M]$, so the group $\calB(R)$ is
$2$-torsion.

We define an element $I\in\calB(R)$ associated to the isogeny class of $A$ as
in~\cite[Def.~5.4]{Howe1995}. Since $\calB(R)$ is $2$-torsion and since $A$ is
isogenous to the square of a simple abelian variety, the formula 
in~\cite[Def.~5.4]{Howe1995} shows that $I = 0$.

Theorem~5.6 of~\cite{Howe1995} says that a class of $\calG(R)$ is attainable if
and only if it is of the form $[M\otimes_{\Rplus} R]$ for an $\Rplus$-module $M$
such that the image of the class $[M]\in\calG(\Rplus)$ under the quotient map
$\calG(\Rplus)\to\calB(R)$ is equal to~$I$. (Note that if a class in $\calG(R)$
is of the form $[M\otimes_{\Rplus}R]$ for an $\Rplus$-module $M$, then the class
$[M] \in \calG(\Rplus)$ is unique. Indeed, the map 
$T\colon\calG(\Rplus) \to \calG(R)$ given by 
$[M] \mapsto [M \otimes_{\Rplus} R]$ is a homomorphism because $R$ is flat 
over~$\Rplus$, and this homomorphism is injective because its composition with 
the norm from $\calG(R)$ to $\calG(\Rplus)$ is the multiplication-by-$2$ map 
on~$\calG(\Rplus).$) To complete our proof of Theorem~\ref{T:genus9}, then, it
will be enough for us to show that both $\Mod_R E[6]$ and $\Mod_R E[12]$ are of
the form $M\otimes_\Rplus R$ for an $\Rplus$-module $M$ whose image in
$\calB(R)$ is \emph{not} equal to $I$; that is, the image of $M$ in $\calB(R)$
should be nonzero.

Let $\qq_3$ be the prime ideal $(3,\pi+\pibar + 1)$ of $R$, let $\qq_2$ be the 
prime ideal $(2,\pi-1) = (2,\pibar)$ of $R$, and let $\qqbar_2$ be the complex 
conjugate $(2,\pibar-1) = (2,\pi)$ of $\qq_2$. We check that $E[6]$ and $E[12]$ 
can both be embedded in~$A$, and that
\begin{align*}
[\Mod_R E[6]]  = [\Mod_R E[3]] + \phantom{2}[\Mod_R E[2]]  &= [R/\qq_3]+ \phantom{2}[R/\qq_2] + \phantom{2}[R/\qqbar_2]\\
[\Mod_R E[12]] = [\Mod_R E[3]] + 2[\Mod_R E[2]]            &= [R/\qq_3]+          2 [R/\qq_2] +          2 [R/\qqbar_2].
\end{align*}

Let $\pp_3$ be the prime ideal $(3, \pi+\pibar + 1)$ of $\Rplus$ and let $\pp_2$
be the prime ideal $(2,\pi+\pibar + 1)$ of $\Rplus$. We check that
\begin{align*}
(\Rplus/\pp_2)\otimes_\Rplus R &\cong R/\qq_2 \oplus R/\qqbar_2\\
(\Rplus/\pp_3)\otimes_\Rplus R &\cong R/\qq_3,
\end{align*}
so \cite[Thm.~5.6]{Howe1995} says that $[\Mod_R E[6]]$ will be attainable if and
only if the image of $[\Rplus/\pp_3] + [\Rplus/\pp_2]\in\calG(\Rplus)$ in
$\calB(R)$ is equal to zero, and that $[\Mod_R E[12]]$ will be attainable if and
only if the image of $[\Rplus/\pp_3] + 2[\Rplus/\pp_2]$ in $\calB(R)$ is zero.

As a first step, we note that $N_{R/\Rplus}([R/\qq_2]) \cong [\Rplus/\pp_2]$, so 
$[\Rplus/\pp_2]$ is in the kernel of the map $\calG(\Rplus)\to\calB(R)$. Thus,
to complete the proof of the theorem, all we must show is that the image of 
$[\Rplus/\pp_3]$ in $\calB(R)$ in nonzero. Our next step, then, is to compute
the group $\calB(R)$.

Let $\OO$ be the maximal order of~$K=\QQ(\pi)$ and let $\OOplus$ be the maximal
order of~$\Kplus$. Computing the discriminants of $\OO$ and $\OOplus$, we find 
that no finite prime of $\OOplus$ ramifies in~$\OO$, so
\cite[Prop.~10.2]{Howe1995} tells us that $\calB(\OO) \cong \Gal(K/\Kplus)$, and
that under this isomorphism the map $\calG(\OOplus)\to\calB(\OO)$ is essentially
the Artin map: if $\frakP$ is a prime of $\OOplus$, then the class of 
$\OOplus/\frakP$ in $\calG(\OOplus)$ is sent to $0$ in $\calB(\OO)$ if and only
if $\frakP$ splits in $K/\Kplus$.

The inclusion map $i\colon R\to \OO$ gives us a pullback map 
$i^*\colon\calB(\OO)\to\calB(R)$, and~\cite[Prop.~10.5]{Howe1995} gives us two 
finite $2$-torsion groups $D_\rms$ and $C_\rms$, maps from these groups to 
$\calB(\OO)$ and $\calB(R)$, respectively, and a map $D_s\to C_s$ such that the
diagram
\[
\xymatrix{
D_\rms\ar[rr]\ar[d]   && \calB(\OO)\ar[d]^{i^*}\\
C_\rms\ar[rr]         && \calB(R)
}
\]
is a pushout diagram.

We see from~\cite[pp.~2386--2387]{Howe1995} that the group $C_\rms$ is an
$\FF_2$-vector space with a basis consisting of symbols $x_\pp$ for the singular
primes $\pp$ of $\Rplus$ that are inert in $R/\Rplus$. We compute that the 
discriminant of $\Rplus$ is $16$ times the discriminant of~$\OOplus$, so the 
only singular prime of $\Rplus$ is the unique prime of $\Rplus$ that lies
over~$2$, namely $\pp_2$. Since $\pp_2$ splits in $R/\Rplus$, the group $C_\rms$
is trivial. 

To compute $D_\rms$ we again use the definition found 
in~\cite[pp.~2386--2387]{Howe1995} together with~\cite[Rem.~10.6]{Howe1995}. We 
find that $D_\rms$ is an $\FF_2$-vector space with a basis consisting of symbols
$y_\frakP$ for the primes $\frakP$ of $\OOplus$ that are inert in $\OO/\OOplus$
and such that $\frakP\cap\Rplus$ lies under a singular prime of~$R$. 
Using~\cite[Prop.~9.4]{Howe1995} we compute that the discriminant of $R$ is 
$2^{16}\cdot 3^6\cdot 11^2$, and since the discriminant of $\OO$ is 
$2^{8}\cdot 3^6\cdot 11^2$, the singular primes of $R$ all lie above~$\pp_2$.
There is only one prime of $\OOplus$ that lies over~$2$, and it splits in
$\OO/\OOplus$. Therefore $D_\rms$ is also trivial.

Since the diagram above is a pushout diagram, we find that the map
$i^*\colon\calB(\OO)\to\calB(R)$ is an isomorphism. 

To compute the image of $\Rplus/\pp_3$ in $\calB(R)$, we use another pushout
diagram that relates $\calB(\OO)$ and $\calB(R)$, namely the one found 
in~\cite[Prop.~10.4]{Howe1995}, which is
\[
\xymatrix{
\calG(\OOplus)/N_{\OO/\OOplus}(\calG(\OO))\ar[rr]\ar[d]^N && \calB(\OO)\ar[d]^{i^*}\\
\calG(\Rplus)/N_{R/\Rplus}(\calG(R))      \ar[rr]         && \calB(R).
}
\]
Here the map $N$ on the left side of the diagram is induced from the map
$\calG(\OOplus)\to\calG(\Rplus)$ that takes the class of a finite
$\OOplus$-module $M$ to the class of $M$ viewed as an $\Rplus$-module.

Let $\frakP_3$ be the prime $(3, \pi+\pibar + 1)$ of $\OOplus$; this is the 
unique prime of $\OOplus$ over~$3$. Since $\Rplus$ is nonsingular at~$3$, the 
restriction of $\frakP_3$ to $\Rplus$ is $\pp_3$, and the map $N$ in the diagram
takes the class of $\OOplus/\frakP_3$ in the upper-left group to the class of 
$\Rplus/\pp_3$ in the lower-left group. Since $i^*$ is an isomorphism, the image
of $\Rplus/\pp_3$ in $\calB(R)$ is nonzero if and only if the image of 
$\OOplus/\frakP_3$ in $\calB(\OO)$ is nonzero. We check that $\frakP_3$ is inert
in $K/\Kplus$, so the image of $\OOplus/\frakP_3$ in $\calB(\OO)$ is nonzero. 
Therefore the image of $\Rplus/\pp_3$ in $\calB(R)$ is nonzero, and the theorem 
is proved.
\end{proof}

%%%%%%%%%%%%%%%%%%%%%%%%%%%%%%%%%%%%%%%%%%%%%%%%%%%%%%%%%%%%%%%%%%%%%%%%%%%%%%%%
%%%%%%%%%%%%%%%%%%%%%%%%%%%%%%%%%%%%%%%%%%%%%%%%%%%%%%%%%%%%%%%%%%%%%%%%%%%%%%%%

\section{Enumerating all excessive curves}
\label{sec:exhaustion}

As we noted in Section~\ref{sec:345}, an excessive curve over $\FF_q$ of 
genus~$3$ is precisely a nonhyperelliptic curve of genus $3$ with no rational 
points (a ``pointless curve''), and a result of Howe, Lauter, and 
Top~\cite[Thm.~1.1]{HLT} says that such curves exist over $\FF_q$ if and only if
$q\le 23$ or $q = 29$ or $q = 32$. In this section we sketch the method we used
to enumerate all such curves over these fields. In total, there are $215$
excessive genus-$3$ curves over these fields; the counts over each finite field
were given in Table~\ref{tab:excessive} in the Introduction.\footnote{
   Note that the curve $(x^2 + xz)^2 + (y^2 + yz)(x^2 + xz) + (y^2 + yz)^2 + z^4 = 0$
   is excessive as a curve over $\FF_2$ \emph{and} as a curve over $\FF_{32}$,
   but we count it twice in our enumeration because we consider curves to be
   schemes over a base field.
}   
Together with the excessive curve of genus~$4$ over $\FF_2$ defined by the two 
equations $x y + z^2 + z w + w^2 = 0$ and $x^3 + y^3 + z^3 + y^2 w + x z w = 0$
in $\PP^4$~\cite[Thm.~5.4]{Faber_Grantham_GF2} and the excessive curve of 
genus~$4$ over $\FF_3$ defined by the two equations $x y + z^2 + w^2 = 0$ and
$x^3 + y^3 + y z^2 + x w^2 - y w^2 - z w^2 = 0$ in 
$\PP^4$~\cite[Thm.~5.3]{Faber_Grantham_GF34}, this gives a total of exactly
$217$ excessive curves over finite fields. 

A human- and computer-readable Magma file, \texttt{AllGenus3.magma}, containing
a list of all $215$ of the excessive genus-$3$ curves over finite fields, is 
available with the other programs associated with this paper.

For $q\le 5$, we used a simple brute force search to find representatives of the
$\PGL(2,\FF_q)$ orbits of pointless plane quartics. We find unique
representatives of each orbit by explicitly calculating the orbits.  We
confirmed our computations for $q=2$ and $q=3$ by checking the LMFDB
database~\cite{LMFDB}, which includes data collected from every curve of
genus~$3$ over $\FF_2$ and~$\FF_3$.

For $7\le q\le 23$, we used a more efficient brute force search that depends on
the curves having enough quadratic points. Note that the Weil bound shows that
if $q\ge 7$, then a genus-$3$ curve over $\FF_q$ has at least 
$q^2 + 1 - 6q\ge 8$ points over $\FF_{q^2}$. If the curve is also pointless 
over~$\FF_q$, then it must have at least $4$ quadratic points, and this will be
enough for our method to work. Let us explain the idea.

Let $s_1$ and $s_2$ be distinct elements of $\FF_{q^2}$ that are conjugate 
over~$\FF_q$. We let $P_1$ be the quadratic point of $\PP^2$ whose geometric
points are $\triple{s_1}01$ and $\triple{s_2}01$, and we let $P_2$ be the 
quadratic point of $\PP^2$ whose geometric points are $\triple{0}{s_1}{1}$ and 
$\triple{0}{s_2}{1}$. If $f\in\FF_q[x,y,z]$ is a homogeneous quartic, we say
that $f$ is \define{pinned} (with respect to $P_1$ and $P_2$) if it vanishes at
$P_1$ and~$P_2$.

Suppose $C_{/\FF_q}$ is an excessive curve of genus~$3$ that has at least three 
quadratic points. We identify $C$ with one of its plane quartic models, so that
$C$ is given by a homogeneous quartic~$f\in \FF_q[x,y,z]$. Each quadratic point
of $C$ determines a rational line in the plane, and at most two quadratic points
of $C$ can lie on a given line because $C$ is defined by a quartic. Since we are
assuming that $C$ has at least three quadratic points, we can choose two such
points that determine distinct rational lines. Then there are exactly four 
automorphisms of~$\PP^2$ that take the two chosen quadratic points to the 
quadratic points $P_1$ and $P_2$, respectively. By applying such an
automorphism, we can assume that $C$ passes through $P_1$ and~$P_2$; that is, we
may assume that $f$ is a pinned quartic. 

The space of plane quartics (up to scalars) is $14$-dimensional. By limiting 
ourselves to pinned quartics, we cut our search space down to $10$ dimensions.
We can further increase our efficiency in looping through this search space by 
using the fact that we only want to consider pointless quartics~$f$. For 
example, since $f(1,0,0)$ and $f(0,1,0)$ and $f(0,0,1)$ are all supposed to be
nonzero, we see that the coefficients of $x^4$, $y^4$, and $z^4$ in $f$ are all
nonzero. Going further: We can choose the collection of coefficients of $x^4$,
$x^3y$, $x^2y^2$, $xy^3$, and $y^4$ all at once, and use the requirement that 
$f(a,b,0)\ne 0$ for all $a, b\in\FF_q$ to limit the possible values of these 
coefficients. We implemented our search using these ideas as often as possible.

Suppose $f \in \FF_q[x,y,z]$ is a homogeneous quartic that defines an excessive
curve~$C$. It is a straightforward matter to produce a list of all pinned
quartics $g$ that define a curve isomorphic to~$C$: Simply loop through all of
the pairs $(Q_1,Q_2)$ of non-colinear quadratic points on~$C$, and for each such
pair write down the four pinned quartics obtained by moving $Q_1$ to $P_1$ and
$Q_2$ to~$P_2$. There are at most $(q^2 + 6q + 1)/2$ quadratic points on~$C$,
and so there are roughly $q^4$ pinned quartics defining~$C$. For the prime
powers $q$ that we are considering, we can easily list all of these quartics.
If we order them --- say, by number of nonzero terms, and lexicographically
within that grouping --- we can then identify the first one. This ``first pinned
quartic'' is a well-defined normal form for~$C$ that allows us to quickly test
whether two pointless quartics are isomorphic to one another.

The basic idea of our enumeration strategy is to run through all pointless
pinned quartics $f$ using the method described above. If such an $f$ defines a
curve of genus~$3$ (that is, if $f$ is geometrically irreducible and defines a
nonsingular curve), then we compute its ``first pinned quartic'' normal form and
add this to our list of excessive curves, if it is not already on the list.

We wrote Magma code to implement this algorithm, and ran it for all $q\le 16$.
Running in Magma 2.26-10 on a $3.2$GHz Apple M1~Max, the case $q = 16$ took just
under $90$ core-hours. To double-check our work and to reach the larger values
of~$q$, we rewrote the code in C as well, and ran it on a cluster of Intel Xeon
E5-2699v3 CPUs. This allowed us to reach all $q \le 23$. See
Table~\ref{tab:quartic_timing} for resources and timing for this computation.

\begin{table}[h!]
\caption{Search resources and wall time for finding excessive plane quartic
  curves.}
\label{tab:quartic_timing}
\begin{tabular}{r@{\qquad}cccr@{\qquad}cc}
  \toprule
  $q$ &  CPUs & Wall time (m)  &\hbox to 1em{}&  $q$ & CPUs & Wall time (m)  \\
  \cmidrule(lr){1-3}                \cmidrule(r){5-7}
   $7$ &  $1$ &  \pz$0$  &&    $16$ & \pz$ 1$ &    $281$  \\    
   $8$ &  $1$ &  \pz$0$  &&    $17$ &    $16$ &  \pz$90$  \\
   $9$ &  $1$ &  \pz$1$  &&    $19$ &    $64$ &    $113$  \\
  $11$ &  $1$ &  \pz$7$  &&    $23$ &    $64$ &    $778$  \\  
  $13$ &  $1$ &    $42$  &&         &         &           \\
  \bottomrule
\end{tabular}
\end{table}

We estimate that it would take about $10{,}000$ CPU-hours to do this calculation
for $q=29$, and about $30{,}000$ CPU-hours for $q=32$. Fortunately, there are
other methods available for these cases.

For $q = 29$, we use Lauter's algorithm to show that if there is a genus-$3$
curve $C_{/\FF_{29}}$ with no points, then its real Weil polynomial is 
$(T-10)^3$. Suppose there is such a $C$, and let $J$ be its Jacobian. There are
two elliptic curves over $\FF_{29}$ that have real Weil polynomial $T-10$:
the curve $E_1\colon y^2 =x^3 + x$, whose endomorphism ring  is isomorphic
to $\ZZ[i]$, where $i^2 = -1$, and the curve $E_2\colon y^2 = x^3 + 12x^2 + 4x$,
whose endomorphism ring isomorphic to $\ZZ[2i]$. From~\cite[Thm.~2]{Kani2011}
we see that $J$ is isomorphic (as an abelian threefold without polarization) to 
one of the products
\[
E_1\times E_1\times E_1, \quad E_1\times E_1\times E_2, \quad
E_1\times E_2\times E_2, \quad \text{and}\quad E_2\times E_2\times E_2.
\]
Let $A \colonequals E_1\times E_1\times E_1.$ Then there is an isogeny 
$\varphi\colon A\to J$ of degree $1$, $2$, $4$, or $8$, depending on which of 
the four products $J$ is isomorphic to. The principal polarization on $J$ pulls
back to a polarization $\lambda$ on $A$ of degree $1$, $4$, $16$, or~$64$. This 
polarization can be represented by a Hermitian matrix $M\in M_3(\ZZ[i])$ of 
determinant $1$, $2$, $4$, or~$8$. Then~\cite[Table~2, p.~190]{HoweLauter2012}
shows that there is an embedding $\psi\colon E_1\to A$ such that $\psi^*\lambda$
is either the principal polarization on $E_1$ or twice the principal 
polarization on~$E_1$. We see from~\cite[Lemma~4.3]{HoweLauter2012} that then
there must be a map $C\to E_1$ of degree $1$ or~$2$. A degree-$1$ map would be
impossible, so $C$ must be a double cover of~$E_1$. It is a straightforward 
matter to enumerate all of the genus-$3$ double covers of~$E_1$; see 
\cite[\S6.1]{Howe_Lauter_2003} for a detailed description of the similar
calculation of all genus-$4$ double covers of an elliptic curve over a finite
field. We find that up to isomorphism, there is exactly one pointless plane
quartic over~$\FF_{29}$, already listed in~\cite[Table~2]{HLT}: the curve 
$x^4 + y^4 + z^4 = 0$.

For $q=32$, we note that it was already proven in~\cite[\S3.3]{Elkies1999} that
there is exactly one pointless plane quartic over~$\FF_{32}$ (and that this
curve is the reduction modulo~$2$ of a twist of the Klein quartic).

%%%%%%%%%%%%%%%%%%%%%%%%%%%%%%%%%%%%%%%%%%%%%%%%%%%%%%%%%%%%%%%%%%%%%%%%%%%%%%%%
%%%%%%%%%%%%%%%%%%%%%%%%%%%%%%%%%%%%%%%%%%%%%%%%%%%%%%%%%%%%%%%%%%%%%%%%%%%%%%%%

% Corrigendum to HoweLauter2003, nocited here because it is cited as a cross
% reference in the bibliography but is not cited directly in paper proper.
\nocite{HoweLauter2007} 

\bibliographystyle{hplaindoi}
\bibliography{excessive}

\providecommand\biburl[1]{\texttt{#1}}
\providecommand{\bysame}{\leavevmode\hbox to3em{\hrulefill}\thinspace}
\providecommand{\xxMR}[2]{\relax\ifhmode\unskip\space\fi
  \href{http://www.ams.org/mathscinet-getitem?mr=#2}{MR~#1}}
\providecommand{\xxZBL}[1]{\relax\ifhmode\unskip\space\fi
  \href{http://www.emis.de/cgi-bin/MATH-item?#1}{ZBL~#1}}
\providecommand{\xxJFM}[1]{\relax\ifhmode\unskip\space\fi
  \href{http://www.emis.de/cgi-bin/JFM-item?#1}{JFM~#1}}
\providecommand{\xxARXIV}[2]{\relax\ifhmode\unskip\space\fi
  \href{http://arxiv.org/abs/#2}{arXiv:#1}}
\providecommand\bibmarginpar{\leavevmode\marginpar}
\providecommand{\href}[2]{#2}
\begin{thebibliography}{10}

\bibitem{magma}
Wieb Bosma, John Cannon, and Catherine Playoust,
  \href{http://dx.doi.org/10.1006/jsco.1996.0125} {\emph{The {M}agma algebra
  system. {I}. {T}he user language}}, J. Symbolic Comput. \textbf{24} (1997),
  no.~3-4, 235--265, Computational algebra and number theory (London, 1993).
  Software available at \url{http://magma.maths.usyd.edu.au/}.
  \xxMR{1484478}{1484478}

\bibitem{Deligne1969}
Pierre Deligne, \href{http://dx.doi.org/10.1007/BF01406076}
  {\emph{Vari\'{e}t\'{e}s ab\'{e}liennes ordinaires sur un corps fini}},
  Invent. Math. \textbf{8} (1969), 238--243. \xxMR{254059}{254059}

\bibitem{Elkies1999}
Noam~D. Elkies, \href{http://www.worldcat.org/oclc/716894561} {\emph{The
  {K}lein quartic in number theory}}, The eightfold way, Math. Sci. Res. Inst.
  Publ., vol.~35, Cambridge Univ. Press, Cambridge, 1999, pp.~51--101.
  \xxMR{1722413}{1722413}

\bibitem{Faber_Grantham_GF34}
Xander Faber and Jon Grantham,
  \href{http://dx.doi.org/10.1080/10586458.2021.1926015} {\emph{Ternary and
  quaternary curves of small fixed genus and gonality with many rational
  points}}, Exper. Math. (2021).

\bibitem{Faber_Grantham_GF2}
\bysame, \href{http://dx.doi.org/10.1016/j.jalgebra.2022.01.008} {\emph{Binary
  curves of small fixed genus and gonality with many rational points}}, J.
  Algebra \textbf{597} (2022), 24--46.

\bibitem{flint-2.6.0}
William Hart, Fredrik Johansson, and Sebastian Pancratz, \emph{{FLINT}: {F}ast
  {L}ibrary for {N}umber {T}heory}, 2020, Version 2.6.0,
  \url{http://flintlib.org}.

\bibitem{Howe1995}
Everett~W. Howe, \href{http://dx.doi.org/10.2307/2154828} {\emph{Principally
  polarized ordinary abelian varieties over finite fields}}, Trans. Amer. Math.
  Soc. \textbf{347} (1995), no.~7, 2361--2401. \xxMR{1297531}{1297531}

\bibitem{Howe_isomorphic_jacobians}
\bysame, \href{http://dx.doi.org/10.1006/jnth.1996.0026} {\emph{Constructing
  distinct curves with isomorphic {J}acobians}}, J. Number Theory \textbf{56}
  (1996), no.~2, 381--390. \xxMR{1373560}{1373560}

\bibitem{Howe_Q_jacobians}
\bysame, \href{http://dx.doi.org/10.1112/S0024610705006812} {\emph{Infinite
  families of pairs of curves over {$\mathbb{Q}$} with isomorphic
  {J}acobians}}, J. London Math. Soc. (2) \textbf{72} (2005), no.~2, 327--350.
  \xxMR{2156657}{2156657}

\bibitem{Howe_Serre_proceedings}
\bysame, \emph{Deducing information about curves over finite fields from their
  {W}eil polynomials}, 2021. \xxARXIV{2110.04221 [math.NT]}{2110.04221}

\bibitem{Howe_Lauter_2003}
Everett~W. Howe and Kristin~E. Lauter,
  \href{http://aif.cedram.org/item?id=AIF_2003__53_6_1677_0} {\emph{Improved
  upper bounds for the number of points on curves over finite fields}}, Ann.
  Inst. Fourier (Grenoble) \textbf{53} (2003), no.~6, 1677--1737, corrigendum
  in \cite{HoweLauter2007}. \xxMR{2038778}{2038778}

\bibitem{HoweLauter2007}
\bysame, \href{http://aif.cedram.org/item?id=AIF_2007__57_3_1019_0}
  {\emph{Corrigendum to: ``{I}mproved upper bounds for the number of points on
  curves over finite fields''}}, Ann. Inst. Fourier (Grenoble) \textbf{57}
  (2007), no.~3, 1019--1021. \xxMR{2336837}{2336837}

\bibitem{HoweLauter2012}
\bysame, \href{http://dx.doi.org/10.4171/119-1/12} {\emph{New methods for
  bounding the number of points on curves over finite fields}}, Geometry and
  arithmetic, EMS Ser. Congr. Rep., Eur. Math. Soc., Z\"{u}rich, 2012,
  pp.~173--212. \xxMR{2987661}{2987661}

\bibitem{HLT}
Everett~W. Howe, Kristin~E. Lauter, and Jaap Top,
  \href{http://www.worldcat.org/oclc/644861796} {\emph{Pointless curves of
  genus three and four}}, Arithmetic, geometry and coding theory ({AGCT} 2003),
  S\'{e}min. Congr., vol.~11, Soc. Math. France, Paris, 2005, pp.~125--141.
  \xxMR{2182840}{2182840}

\bibitem{Kani2011}
Ernst Kani, \href{http://dx.doi.org/10.1007/s13348-010-0029-1} {\emph{Products
  of {CM} elliptic curves}}, Collect. Math. \textbf{62} (2011), no.~3,
  297--339. \xxMR{2825715}{2825715}

\bibitem{Lauter_algorithm_1998}
Kristin Lauter, \href{http://dx.doi.org/10.1007/978-3-642-57189-3_15}
  {\emph{Zeta functions of curves over finite fields with many rational
  points}}, Coding theory, cryptography and related areas ({G}uanajuato, 1998),
  Springer, Berlin, 2000, pp.~167--174. \xxMR{1749457}{1749457}

\bibitem{LMFDB}
The {LMFDB Collaboration}, \emph{The {L}-functions and modular forms database},
  \url{http://www.lmfdb.org}, accessed 13 November 2021.

\bibitem{Manin1963}
Yuri~I. Manin, \href{http://dx.doi.org/10.1070/RM1963v018n06ABEH001142}
  {\emph{Theory of commutative formal groups over fields of finite
  characteristic}}, Russian Math. Surveys \textbf{18} (1963), no.~6, 1--83,
  translated by P.~M.~Cohn. \xxMR{0157972}{0157972}

\bibitem{Poonen_gonality}
Bjorn Poonen, \href{http://dx.doi.org/10.4310/MRL.2007.v14.n4.a14}
  {\emph{Gonality of modular curves in characteristic {$p$}}}, Math. Res. Lett.
  \textbf{14} (2007), no.~4, 691--701. \xxMR{2335995}{2335995}

\bibitem{Robinson1964}
Raphael~M. Robinson, \href{http://dx.doi.org/10.2307/2002941} {\emph{Algebraic
  equations with span less than {$4$}}}, Math. Comp. \textbf{18} (1964),
  547--559. \xxMR{169374}{169374}

\bibitem{sage_9.1}
The {Sage Developers}, \emph{{S}agemath, the {S}age {M}athematics {S}oftware
  {S}ystem ({V}ersion 9.1)}, 2020, \url{https://www.sagemath.org}.

\bibitem{Serre_Rational_Points_Book}
Jean-Pierre Serre, \href{http://www.worldcat.org/oclc/1242402661}
  {\emph{Rational points on curves over finite fields}}, Documents
  Math\'{e}matiques (Paris), vol.~18, Soci\'{e}t\'{e} Math\'{e}matique de
  France, Paris, 2020, with contributions by Everett Howe, Joseph Oesterl\'e,
  and Christophe Ritzenthaler. \xxMR{4242817}{4242817}

\bibitem{Tate_Abelian_Varieties_Finite_Fields_1966}
John Tate, \href{http://dx.doi.org/10.1007/BF01404549} {\emph{Endomorphisms of
  abelian varieties over finite fields}}, Invent. Math. \textbf{2} (1966),
  134--144. \xxMR{206004}{206004}

\bibitem{Tate1971}
\bysame, \href{http://dx.doi.org/10.1007/BFb0058807} {\emph{Classes
  d'isog\'{e}nie des vari\'{e}t\'{e}s ab\'{e}liennes sur un corps fini
  (d'apr\`es {T}. {H}onda)}}, S\'{e}minaire {B}ourbaki. {V}ol. 1968/69:
  {E}xpos\'{e}s 347--363, Lecture Notes in Math., vol. 175, Springer, Berlin,
  1971, pp.~Exp. No. 352, 95--110. \xxMR{3077121}{3077121}

\bibitem{TVN_alg_geom_codes}
Michael Tsfasman, Serge Vl\u{a}du\c{t}, and Dmitry Nogin,
  \href{http://dx.doi.org/10.1090/surv/139} {\emph{Algebraic geometric codes:
  basic notions}}, Mathematical Surveys and Monographs, vol. 139, American
  Mathematical Society, Providence, RI, 2007. \xxMR{2339649}{2339649}

\bibitem{hackers}
Henry~S. Warren, \href{http://www.worldcat.org/oclc/912380039} {\emph{Hacker's
  delight}}, Addison-Wesley Longman Publishing Co., Inc., USA, 2002.

\bibitem{Waterhouse1979}
William~C. Waterhouse, \href{http://dx.doi.org/10.1007/978-1-4612-6217-6}
  {\emph{Introduction to affine group schemes}}, Graduate Texts in Mathematics,
  vol.~66, Springer-Verlag, New York, 1979. \xxMR{547117}{547117}

\end{thebibliography}

%%%%%%%%%%%%%%%%%%%%%%%%%%%%%%%%%%%%%%%%%%%%%%%%%%%%%%%%%%%%%%%%%%%%%%%%%%%%%%%%
%%%%%%%%%%%%%%%%%%%%%%%%%%%%%%%%%%%%%%%%%%%%%%%%%%%%%%%%%%%%%%%%%%%%%%%%%%%%%%%%
\appendix

\section{Polynomial Bases for Proposition~\ref{prop:genus7_special}}
\label{sec:bases}
For completeness, we give the bases of polynomials used to search for singular
plane curves of genus~$7$ as described by Proposition~\ref{prop:genus7_special}.
To unify the presentation, we formalize the properties these bases must have in 
order to efficiently construct curves that do not pass through a rational point 
of the plane.

\begin{formalism}
  \label{formal}
  Let $S_1, \ldots, S_5$ be five sets of homogeneous polynomials of degree~$9$
  in $\FF_2[x,y,z]$ that satisfy the following hypotheses\textup{:}
\begin{itemize}
\item At each of the three points $\triple{1}{0}{0}$, $\triple{0}{1}{0}$, and 
      $\triple{0}{0}{1}$, exactly one element of $S_1$ is nonzero. Each element
      of $S_i$ with $i > 1$ vanishes at these three points.
\item At the point $P = \triple{1}{1}{0}$, each polynomial in $S_2$ is nonzero. 
      Each polynomial in $S_3, S_4$, and $S_5$ vanishes at $P$. An even number
      of elements of $S_1$ are nonzero at~$P$.
\item At the point $P = \triple{1}{0}{1}$, each polynomial in $S_3$ is nonzero.
      Each polynomial in $S_2, S_4$, and $S_5$ vanishes at $P$. An even number
      of elements of $S_1$ are nonzero at~$P$.
\item At the point $P = \triple{0}{1}{1}$, each polynomial in $S_4$ is nonzero. 
      Each polynomial in $S_2, S_3$, and $S_5$ vanishes at $P$. If $S_4$ is 
      nonempty (resp., empty), an even (resp., odd) number of elements of $S_1$
      are nonzero at~$P$.
\item At the point $P = \triple{1}{1}{1}$, each polynomial in $S_i$ is nonzero
      for $i  = 2, 3, 4, 5$. If $S_4$ is nonempty (resp., empty), an odd (resp., 
      even) number of elements of $S_1$ are nonzero at~$P$.
\end{itemize}
Let $F_1$ be the sum of all polynomials in $S_1$. For $i = 2, 3, 4, 5$, let 
$F_i$ be the sum of an odd number of polynomials in $S_i$. If 
$S_4 = \varnothing$, we take $F_4 = 0$. Then $F = \sum F_i$ does not vanish at
any rational point of the plane.
\end{formalism}

\begin{remark}
  This formalism applies to our constructions in Sections~\ref{sec:genus6}
  and~\ref{sec:genus7_4}. For $g = 6$ and $g=7$, we used
  \begin{align*}
    S_1 &= \{x^{g+1}, y^{g+1}, z^{g+1}\} \\
    S_2 &= \{x^i y^{g+1-i} \colon 1 \le i \le g\} \\
    S_3 &= \{x^i z^{g+1-i} \colon 1 \le i \le g\} \\
    S_4 &= \{y^i z^{g+1-i} \colon 1 \le i \le g\} \\
    S_5 &= \{x^i y^j z^{g+1-i-j} \colon 1 \le i \le g-1,\ 1 \le j \le g-i\}. \\
  \end{align*}
\end{remark}

%%%%%%%%%%%%%%%%%%%%%%%%%%%%%%%%%%%%%%%%%%%%%%%%%%%%%%%%%%%%%%%%%%%%%%%%%%%%%%%%

\subsubsection{Triple cubic point on a line}

In the language of Formalism~\ref{formal}, we take the following to be our basis
for the space of degree-$9$ homogeneous polynomials that vanish to order~$3$ at 
the cubic point $\triple{0}{1}{t}$:
\begin{align*}
S_1 &= \{x^9,\  (y^3 + y^2 z + z^3)^3\} \\
S_2 &= \{x^i y^{9-i} : 3 \le i \le 8\} \cup
       \{x^2 y^4 (y^3 + y^2 z + z^3), \ x y^2 (y^3 + y^2 z + z^3)^2\}\\
S_3 &= \{x^i z^{9-i} : 3 \le i \le 8\} \cup
       \{x^2 z (y^3 + y^2 z + z^3)^2, \ x z^2 (y^3 + y^2 z + z^3)^2\} \\
S_4 &= \varnothing\phantom{ \{z^{9-i}\}} \\
S_5 &= \{x^i y^j z^{9-i-j} : 3 \le i \le 7, \ 1 \le j \le 5 , \ i + j < 9\}\\
    & \qquad \cup \{x y z (y^3 + y^2 z + z^3)^2, \ x^2 y z (y^5 + y z^4 + z^5), \\ 
    & \qquad\qquad  x^2 y^2 z^2 (y^3 + y^2 z + z^3), \ x^2 y^3 z (y^3 + y^2 z + z^3)\}.
\end{align*}
One verifies directly that all of these polynomials vanish to order~$3$ at 
$\triple{0}{1}{t}$. Each polynomial in $S\colonequals S_1 \cup \cdots \cup S_5$, 
except for the element $x^2 y^2 z^2 (y^3 + y^2 z + z^3)$ of $S_5$, has the 
property that it contains a monomial that is not present in any other polynomial
in~$S$. Consequently, the full set of $37$ polynomials in $S$ must be linearly 
independent over~$\FF_2$.

%%%%%%%%%%%%%%%%%%%%%%%%%%%%%%%%%%%%%%%%%%%%%%%%%%%%%%%%%%%%%%%%%%%%%%%%%%%%%%%%

\subsubsection{Double and single cubic points on a line}

In the language of Formalism~\ref{formal}, we take the following to be our basis
for the space of degree-$9$ homogeneous polynomials that vanish at the cubic
points $\triple{0}{1}{t}$ and $\triple{0}{1}{t+1}$, with order at least two at the former:
\begin{align*}
S_1 &= \{x^9, \ (y^3 + y z^2 + z^3) (y^3 + y^2 z + z^3)^2 \} \\
S_2 &= \{x^i y^{9-i} : 2\le i\le 8\} \cup \{x y^5 (y^3 + y^2 z + z^3) \}\\
S_3 &= \{x^i z^{9-i} : 2\le i\le 8\} \cup \{x z^2 (y^3 + y^2 z + z^3)^2 \}\\
S_4 &= \varnothing\phantom{ \{z^{9-i}\}} \\
S_5 &= \{x^i y^j z^{9-i-j} : 2\le i\le 7, \ 1\le j\le 6, \   i + j < 9 \} \\
    &\qquad \cup   \{x y z (y^3 + y^2 z + z^3)^2, x y^2 z (y^5 + y z^4 + z^5),\\
    &\qquad\qquad   x y^3 z^2 (y^3 + y^2 z + z^3), x y^4 z (y^3 + y^2 z + z^3)\}.
\end{align*}

One verifies directly that all of these polynomials vanish at $\triple{0}{1}{t}$ and
at $\triple{0}{1}{t+1}$ to the correct orders.  Each polynomial in 
$S\colonequals S_1 \cup \cdots \cup S_5$, except for 
$x y^4 z (y^3 + y^2 z + z^3)$ and $x y^3 z^2 (y^3 + y^2 z + z^3)$, has the 
property that it contains a monomial that is not present in any other polynomial
in~$S$. Therefore, every nontrivial linear relation among the polynomials in $S$
must involve only the two exceptional polynomials, and since those two 
polynomials are distinct, the full set of $43$ polynomials in $S$ must be 
linearly independent over~$\FF_2$.

%%%%%%%%%%%%%%%%%%%%%%%%%%%%%%%%%%%%%%%%%%%%%%%%%%%%%%%%%%%%%%%%%%%%%%%%%%%%%%%%

\subsubsection{Cubic point on two distinct lines}

In the language of Formalism~\ref{formal}, we take the following to be our basis
for the space of degree-$9$ homogeneous polynomials that vanish at the cubic
points $\triple{0}{1}{t}$ and $\triple{1}{0}{t}$:
\begin{align*}
S_1 &= \{x^2  (x^7 + z^7) , \  y^2 (y^7 + z^7) , \ z^2 (x^7 + y^7 + z^7)\} \\
S_2 &= \{x^i y^{9-i} : 1 \le i \le 8 \}\\
S_3 &= \{x^2 z (x^3 + x^2 z + z^3)^2, \ x^2 z^2 (x^5 + x z^4 + z^5), \ x z^3 (x^5 + x z^4 + z^5), \\
    & \qquad x^3 z^3 (x^3 + x^2 z + z^3), \ x z^5 (x^3 + x^2 z + z^3)\}\\
S_4 &= \{y^2 z (y^3 + y^2 z + z^3)^2, \ y^2 z^2 (y^5 + y z^4 + z^5),  \ y z^3 (y^5 + y z^4 + z^5), \\
    & \qquad y^3 z^3 (y^3 + y^2 z + z^3), \ y z^5 (y^3 + y^2 z + z^3)\} \\
S_5 &= \{x^i  y^j  z^{9-i-j} : 1 \le i \le 7, \ 1 \le j \le 7, \  i + j < 9\}.
\end{align*}

One verifies directly that all of these polynomials vanish at $\triple{0}{1}{t}$ and
$\triple{1}{0}{t}$. Every polynomial in $S\colonequals S_1 \cup \cdots \cup S_5$,
except for those in 
\[
T\colonequals \{x z^3 (x^5 + x z^4 + z^5), \ x^2 z^2 (x^5 + x z^4 + z^5), \ 
                  y z^3 (y^5 + y z^4 + z^5), \ y^2 z^2 (y^5 + y z^4 + z^5)\},
\]
has the property that it contains a monomial that is not present in any other 
polynomial in~$S$. Therefore, if there is a nontrivial linear relation among the 
polynomials in $S$, it must involve only the polynomials in~$T$. But now we note
that every polynomial in $T$ includes a monomial that appears in no other 
element of $T$, so the elements of~$T$ --- and hence the elements of~$S$ ---
are linearly independent over~$\FF_2$.

%%%%%%%%%%%%%%%%%%%%%%%%%%%%%%%%%%%%%%%%%%%%%%%%%%%%%%%%%%%%%%%%%%%%%%%%%%%%%%%%
%%%%%%%%%%%%%%%%%%%%%%%%%%%%%%%%%%%%%%%%%%%%%%%%%%%%%%%%%%%%%%%%%%%%%%%%%%%%%%%%

\section{Isogeny Classes for Genus~7}
\label{sec:genus7_classes}

Here we collect the real Weil polynomials for potentially excessive curves of
genus~$7$ over $\FF_2$ that survived Lauter's algorithm; 
see Theorem~\ref{thm:Lauter_search}.

{\setlength{\LTcapwidth}{0.9\textwidth}
\renewcommand{\arraystretch}{1.1} % Give the polynomials more breathing room
\begin{longtable}{clrrrrrrr}
\caption{Real Weil polynomials for possible excessive curves of genus~$7$ over $\FF_2$ (as found by Lauter's algorithm) and the corresponding numbers of places.\label{table:genus7}}\\
\toprule
    No. & Factored real Weil polynomial & $a_1$ & $a_2$ & $a_3$& $a_4$& $a_5$& $a_6$& $a_7$  \\      
\midrule
\endfirsthead
\caption{(continued)}\\
\toprule
    No. & Factored real Weil polynomial & $a_1$ & $a_2$ & $a_3$& $a_4$& $a_5$& $a_6$& $a_7$  \\      
\midrule
\endhead
\bottomrule \endfoot
\pz1 & $T^{7} - 3 T^{6} - 12 T^{5} + 36 T^{4} + 44 T^{3} - 132 T^{2} - 44 T + 141$ & 0 & 0 & 0 & 8 & 0 & 4 & 9 \\
\pz2 & $T^{7} - 3 T^{6} - 12 T^{5} + 38 T^{4} + 34 T^{3} - 132 T^{2} + 13 T + 78$ & 0 & 0 & 2 & 4 & 0 & 6 & 1 \\
\pz3 & $T  (T^{6} - 3 T^{5} - 12 T^{4} + 39 T^{3} + 27 T^{2} - 126 T + 57)$ & 0 & 0 & 3 & 0 & 0 & 13 & 9 \\
\pz4 & $T^{7} - 3 T^{6} - 12 T^{5} + 40 T^{4} + 24 T^{3} - 132 T^{2} + 75 T + 1$ & 0 & 0 & 4 & 0 & 0 & 9 & 10 \\
\pz5 & $(T^{3} - 4 T^{2} + 3 T + 1)  (T^{4} + T^{3} - 11 T^{2} - 8 T + 25)$ & 0 & 0 & 4 & 1 & 0 & 5 & 6 \\
\pz6 & $(T^{3} - T^{2} - 7 T + 6)  (T^{4} - 2 T^{3} - 7 T^{2} + 14 T - 2)$ & 0 & 0 & 5 & 0 & 0 & 9 & 12 \\
\pz7 & $T^{7} - 3 T^{6} - 11 T^{5} + 33 T^{4} + 33 T^{3} - 99 T^{2} - 14 T + 52$ & 0 & 1 & 0 & 4 & 0 & 9 & 10 \\
\pz8 & $(T + 2)  (T^{6} - 5 T^{5} - T^{4} + 35 T^{3} - 36 T^{2} - 30 T + 38)$ & 0 & 1 & 0 & 5 & 0 & 4 & 10 \\
\pz9 & $T^{7} - 3 T^{6} - 11 T^{5} + 33 T^{4} + 35 T^{3} - 105 T^{2} - 25 T + 86$ & 0 & 1 & 0 & 6 & 0 & 4 & 11 \\
10   & $T^{7} - 3 T^{6} - 11 T^{5} + 33 T^{4} + 35 T^{3} - 105 T^{2} - 24 T + 83$ & 0 & 1 & 0 & 6 & 0 & 5 & 11 \\
11   & $(T^{3} - 2 T^{2} - 5 T + 8)  (T^{4} - T^{3} - 8 T^{2} + 4 T + 12)$ & 0 & 1 & 0 & 7 & 0 & 4 & 12 \\
12   & $T^{7} - 3 T^{6} - 10 T^{5} + 30 T^{4} + 24 T^{3} - 72 T^{2} + 11$ & 0 & 2 & 0 & 1 & 0 & 10 & 11 \\
13   & $(T^{2} - T - 5)  (T^{5} - 2 T^{4} - 7 T^{3} + 13 T^{2} + 3 T - 7)$ & 0 & 2 & 0 & 2 & 0 & 4 & 11 \\
14   & $T^{7} - 3 T^{6} - 10 T^{5} + 30 T^{4} + 25 T^{3} - 75 T^{2} - 3 T + 21$ & 0 & 2 & 0 & 2 & 0 & 9 & 12 \\
15   & $T^{7} - 3 T^{6} - 10 T^{5} + 30 T^{4} + 25 T^{3} - 75 T^{2} + 2 T + 7$ & 0 & 2 & 0 & 2 & 0 & 14 & 13 \\
16   & $(T - 1)  (T^{2} - T - 5)  (T^{4} - T^{3} - 8 T^{2} + 5 T + 9)$ & 0 & 2 & 0 & 3 & 0 & 3 & 12 \\
17   & $T^{7} - 3 T^{6} - 10 T^{5} + 30 T^{4} + 26 T^{3} - 78 T^{2} - 6 T + 31$ & 0 & 2 & 0 & 3 & 0 & 8 & 13 \\
18   & $T^{7} - 3 T^{6} - 10 T^{5} + 30 T^{4} + 26 T^{3} - 78 T^{2} - 5 T + 28$ & 0 & 2 & 0 & 3 & 0 & 9 & 13 \\
19   & $(T^{2} - T - 5)  (T^{5} - 2 T^{4} - 7 T^{3} + 13 T^{2} + 5 T - 11)$ & 0 & 2 & 0 & 4 & 0 & 2 & 13 \\
20   & $T^{7} - 3 T^{6} - 10 T^{5} + 30 T^{4} + 27 T^{3} - 81 T^{2} - 13 T + 52$ & 0 & 2 & 0 & 4 & 0 & 3 & 13 \\
21   & $T^{7} - 3 T^{6} - 10 T^{5} + 30 T^{4} + 27 T^{3} - 81 T^{2} - 9 T + 41$ & 0 & 2 & 0 & 4 & 0 & 7 & 14 \\
22   & $(T^{2} - 2 T - 2)  (T^{5} - T^{4} - 10 T^{3} + 8 T^{2} + 23 T - 19)$ & 0 & 2 & 0 & 4 & 0 & 8 & 14 \\
23   & $T^{7} - 3 T^{6} - 10 T^{5} + 30 T^{4} + 28 T^{3} - 84 T^{2} - 15 T + 59$ & 0 & 2 & 0 & 5 & 0 & 3 & 14 \\
24   & $T^{7} - 3 T^{6} - 10 T^{5} + 30 T^{4} + 28 T^{3} - 84 T^{2} - 12 T + 51$ & 0 & 2 & 0 & 5 & 0 & 6 & 15 \\
25   & $(T + 1)  (T^{3} - 2 T^{2} - 5 T + 8)^{2}$ & 0 & 2 & 0 & 6 & 0 & 4 & 16 \\
26   & $T^{7} - 3 T^{6} - 9 T^{5} + 27 T^{4} + 18 T^{3} - 54 T^{2} + T + 11$ & 0 & 3 & 0 & 0 & 0 & 7 & 14 \\
27   & $T  (T^{2} - T - 5)  (T^{4} - 2 T^{3} - 6 T^{2} + 11 T - 1)$ & 0 & 3 & 0 & 0 & 0 & 11 & 15 \\
28   & $T^{7} - 3 T^{6} - 9 T^{5} + 27 T^{4} + 19 T^{3} - 57 T^{2} - 2 T + 21$ & 0 & 3 & 0 & 1 & 0 & 5 & 15 \\
29   & $T^{7} - 3 T^{6} - 9 T^{5} + 27 T^{4} + 19 T^{3} - 57 T^{2} - T + 18$ & 0 & 3 & 0 & 1 & 0 & 6 & 15 \\
30   & $(T^{2} - T - 5)  (T^{5} - 2 T^{4} - 6 T^{3} + 11 T^{2} - 2)$ & 0 & 3 & 0 & 1 & 0 & 9 & 16 \\
31   & $T^{7} - 3 T^{6} - 9 T^{5} + 27 T^{4} + 19 T^{3} - 57 T^{2} + 3 T + 7$ & 0 & 3 & 0 & 1 & 0 & 10 & 16 \\
32   & $T^{7} - 3 T^{6} - 9 T^{5} + 27 T^{4} + 19 T^{3} - 57 T^{2} + 4 T + 4$ & 0 & 3 & 0 & 1 & 0 & 11 & 16 \\
33   & $(T^{3} - T^{2} - 5 T + 1)  (T^{4} - 2 T^{3} - 6 T^{2} + 10 T + 1)$ & 0 & 3 & 0 & 1 & 0 & 12 & 16 \\
34   & $(T - 1)  (T^{2} - 2 T - 2)  (T^{4} - 9 T^{2} - 2 T + 14)$ & 0 & 3 & 0 & 2 & 0 & 4 & 16 \\
35   & $T^{7} - 3 T^{6} - 9 T^{5} + 27 T^{4} + 20 T^{3} - 60 T^{2} - 3 T + 25$ & 0 & 3 & 0 & 2 & 0 & 5 & 16 \\
36   & $(T^{2} - T - 5)  (T^{5} - 2 T^{4} - 6 T^{3} + 11 T^{2} + T - 4)$ & 0 & 3 & 0 & 2 & 0 & 7 & 17 \\
37   & $T^{7} - 3 T^{6} - 9 T^{5} + 27 T^{4} + 20 T^{3} - 60 T^{2} + 17$ & 0 & 3 & 0 & 2 & 0 & 8 & 17 \\
38   & $T^{7} - 3 T^{6} - 9 T^{5} + 27 T^{4} + 20 T^{3} - 60 T^{2} + T + 14$ & 0 & 3 & 0 & 2 & 0 & 9 & 17 \\
39   & $T^{7} - 3 T^{6} - 9 T^{5} + 27 T^{4} + 20 T^{3} - 60 T^{2} + 2 T + 11$ & 0 & 3 & 0 & 2 & 0 & 10 & 17 \\
40   & $T^{7} - 3 T^{6} - 9 T^{5} + 27 T^{4} + 20 T^{3} - 60 T^{2} + 3 T + 8$ & 0 & 3 & 0 & 2 & 0 & 11 & 17 \\
41  & $(T^{3} - T^{2} - 6 T + 3)  (T^{4} - 2 T^{3} - 5 T^{2} + 7 T + 3)$ & 0 & 3 & 0 & 2 & 0 & 11 & 18 \\
42  & $T^{7} - 3 T^{6} - 9 T^{5} + 27 T^{4} + 20 T^{3} - 60 T^{2} + 4 T + 6$ & 0 & 3 & 0 & 2 & 0 & 12 & 18 \\
43  & $T^{7} - 3 T^{6} - 9 T^{5} + 27 T^{4} + 21 T^{3} - 63 T^{2} - 5 T + 32$ & 0 & 3 & 0 & 3 & 0 & 4 & 17 \\
44  & $(T - 1)  (T^{2} - 2 T - 2)  (T^{2} - T - 5)  (T^{2} + T - 3)$ & 0 & 3 & 0 & 3 & 0 & 5 & 18 \\
45  & $T^{7} - 3 T^{6} - 9 T^{5} + 27 T^{4} + 21 T^{3} - 63 T^{2} - 3 T + 27$ & 0 & 3 & 0 & 3 & 0 & 6 & 18 \\
46  & $T^{7} - 3 T^{6} - 9 T^{5} + 27 T^{4} + 21 T^{3} - 63 T^{2} - 2 T + 24$ & 0 & 3 & 0 & 3 & 0 & 7 & 18 \\
47  & $T^{7} - 3 T^{6} - 9 T^{5} + 27 T^{4} + 21 T^{3} - 63 T^{2} - T + 21$ & 0 & 3 & 0 & 3 & 0 & 8 & 18 \\
48  & $(T^{3} - 6 T - 3)  (T^{4} - 3 T^{3} - 3 T^{2} + 12 T - 6)$ & 0 & 3 & 0 & 3 & 0 & 9 & 18 \\
49  & $T^{7} - 3 T^{6} - 9 T^{5} + 27 T^{4} + 21 T^{3} - 63 T^{2} + 19$ & 0 & 3 & 0 & 3 & 0 & 9 & 19 \\
50  & $T^{7} - 3 T^{6} - 9 T^{5} + 27 T^{4} + 22 T^{3} - 66 T^{2} - 5 T + 34$ & 0 & 3 & 0 & 4 & 0 & 5 & 19 \\
51  & $T^{7} - 3 T^{6} - 9 T^{5} + 27 T^{4} + 22 T^{3} - 66 T^{2} - 4 T + 31$ & 0 & 3 & 0 & 4 & 0 & 6 & 19 \\
52  & $(T - 1)  (T + 2)  (T^{2} - 2 T - 2)  (T^{3} - 2 T^{2} - 5 T + 8)$ & 0 & 3 & 0 & 4 & 0 & 6 & 20 \\
53  & $(T - 1)  (T^{6} - 2 T^{5} - 10 T^{4} + 14 T^{3} + 28 T^{2} - 14 T - 11)$ & 0 & 4 & 0 & 0 & 0 & 7 & 20 \\
54  & $(T^{3} - T^{2} - 6 T + 4)  (T^{4} - 2 T^{3} - 4 T^{2} + 4 T + 2)$ & 0 & 4 & 0 & 0 & 0 & 8 & 20 \\
55  & $T^{7} - 3 T^{6} - 8 T^{5} + 24 T^{4} + 14 T^{3} - 42 T^{2} + 4 T + 9$ & 0 & 4 & 0 & 0 & 0 & 8 & 21 \\
56  & $T^{7} - 3 T^{6} - 8 T^{5} + 24 T^{4} + 14 T^{3} - 42 T^{2} + 5 T + 6$ & 0 & 4 & 0 & 0 & 0 & 9 & 21 \\
57  & $T^{7} - 3 T^{6} - 8 T^{5} + 24 T^{4} + 14 T^{3} - 42 T^{2} + 6 T + 3$ & 0 & 4 & 0 & 0 & 0 & 10 & 21 \\
58  & $T  (T^{6} - 3 T^{5} - 8 T^{4} + 24 T^{3} + 14 T^{2} - 42 T + 7)$ & 0 & 4 & 0 & 0 & 0 & 11 & 21 \\
59  & $(T - 1)  (T^{6} - 2 T^{5} - 10 T^{4} + 14 T^{3} + 29 T^{2} - 16 T - 13)$ & 0 & 4 & 0 & 1 & 0 & 7 & 22 \\
60  & $T^{7} - 3 T^{6} - 8 T^{5} + 24 T^{4} + 15 T^{3} - 45 T^{2} + 4 T + 10$ & 0 & 4 & 0 & 1 & 0 & 8 & 22 \\
61  & $T^{7} - 3 T^{6} - 8 T^{5} + 24 T^{4} + 15 T^{3} - 45 T^{2} + 4 T + 11$ & 0 & 4 & 0 & 1 & 0 & 8 & 23 \\
62  & $T^{7} - 3 T^{6} - 8 T^{5} + 24 T^{4} + 15 T^{3} - 45 T^{2} + 5 T + 7$ & 0 & 4 & 0 & 1 & 0 & 9 & 22 \\
63  & $T^{7} - 3 T^{6} - 8 T^{5} + 24 T^{4} + 15 T^{3} - 45 T^{2} + 5 T + 8$ & 0 & 4 & 0 & 1 & 0 & 9 & 23 \\
64  & $T^{7} - 3 T^{6} - 8 T^{5} + 24 T^{4} + 15 T^{3} - 45 T^{2} + 6 T + 4$ & 0 & 4 & 0 & 1 & 0 & 10 & 22 \\
65  & $(T^{2} - T - 5)  (T^{5} - 2 T^{4} - 5 T^{3} + 9 T^{2} - T - 1)$ & 0 & 4 & 0 & 1 & 0 & 10 & 23 \\
66  & $T^{7} - 3 T^{6} - 8 T^{5} + 24 T^{4} + 15 T^{3} - 45 T^{2} + 7 T + 1$ & 0 & 4 & 0 & 1 & 0 & 11 & 22 \\
67  & $T^{7} - 3 T^{6} - 8 T^{5} + 24 T^{4} + 15 T^{3} - 45 T^{2} + 7 T + 2$ & 0 & 4 & 0 & 1 & 0 & 11 & 23 \\
68  & $T  (T^{3} - 2 T^{2} - 5 T + 8)  (T^{3} - T^{2} - 5 T + 1)$ & 0 & 4 & 0 & 1 & 0 & 12 & 24 \\
69  & $T^{7} - 3 T^{6} - 8 T^{5} + 24 T^{4} + 16 T^{3} - 48 T^{2} + 3 T + 14$ & 0 & 4 & 0 & 2 & 0 & 7 & 23 \\
70  & $(T - 1)^{2}  (T^{2} - T - 5)  (T^{3} - 6 T - 3)$ & 0 & 4 & 0 & 2 & 0 & 7 & 24 \\
71  & $T^{7} - 3 T^{6} - 8 T^{5} + 24 T^{4} + 16 T^{3} - 48 T^{2} + 4 T + 11$ & 0 & 4 & 0 & 2 & 0 & 8 & 23 \\
72  & $(T^{3} - 2 T^{2} - 4 T + 6)  (T^{4} - T^{3} - 6 T^{2} + 2 T + 2)$ & 0 & 4 & 0 & 2 & 0 & 8 & 24 \\
73  & $T^{7} - 3 T^{6} - 8 T^{5} + 24 T^{4} + 16 T^{3} - 48 T^{2} + 4 T + 13$ & 0 & 4 & 0 & 2 & 0 & 8 & 25 \\
74  & $(T + 2)  (T^{6} - 5 T^{5} + 2 T^{4} + 20 T^{3} - 24 T^{2} + 5)$ & 0 & 4 & 0 & 2 & 0 & 9 & 25 \\
75  & $(T - 1)  (T^{6} - 2 T^{5} - 9 T^{4} + 12 T^{3} + 23 T^{2} - 10 T - 1)$ & 0 & 5 & 0 & 0 & 0 & 10 & 28 \\
76  & $T^{7} - 3 T^{6} - 7 T^{5} + 21 T^{4} + 11 T^{3} - 33 T^{2} + 9 T + 2$ & 0 & 5 & 0 & 0 & 0 & 10 & 29 \\
77  & $(T - 1)  T  (T + 2)  (T^{2} - 3 T + 1)  (T^{2} - T - 5)$ & 0 & 5 & 0 & 0 & 0 & 11 & 30 \\
78  & $T^{7} - 3 T^{6} - 7 T^{5} + 21 T^{4} + 11 T^{3} - 33 T^{2} + 10 T + 1$ & 0 & 5 & 0 & 0 & 0 & 11 & 31 \\
79  & $(T - 1)  (T^{3} - T^{2} - 5 T + 1)^{2}$ & 0 & 5 & 0 & 0 & 0 & 12 & 32 \\
\end{longtable}
}

\end{document}